\documentclass[12pt]{article}
\usepackage{amsmath}
\usepackage{amsmath,amssymb, amsthm, latexsym, amsfonts, epsfig, color,graphicx,}
\usepackage{mathrsfs}
\usepackage{indentfirst}
\usepackage{chngcntr}

\begin{document}
\renewcommand{\a}{\alpha}
\newcommand{\D}{\Delta}
\newcommand{\ddt}{\frac{d}{dt}}
\counterwithin{equation}{section}
\newcommand{\e}{\epsilon}
\newcommand{\eps}{\varepsilon}
\newtheorem{theorem}{Theorem}[section]

\newtheorem{proposition}{Proposition}[section]
\newtheorem{lemma}[proposition]{Lemma}
\newtheorem{remark}{Remark}[section]
\newtheorem{example}{Example}[section]
\newtheorem{definition}{Definition}[section]
\newtheorem{corollary}{Corollary}[section]
\makeatletter
\newcommand{\rmnum}[1]{\romannumeral #1}
\newcommand{\Rmnum}[1]{\expandafter\@slowromancap\romannumeral #1@}
\makeatother

\title{\bf Dispersive decay for the  Nonlinear magnetic Schr\"{o}dinger equation\footnote{This work was supported by  National Natural Science Foundation of China (61671009, 12171178).}}
\author{Lei Wei,\,\,  Zhiwen Duan\footnote{Corresponding author.}\\
{\small {\it  School of Mathematics and Statistics, Huazhong University of Science }}\\
{\small {\it and Technology, Wuhan, {\rm 430074,} P.R.China}}  \\
{\small {\it Email: d201880013@hust.edu.cn}}\\
{\small {\it Email: duanzhw@hust.edu.cn}}}
\date{}
\maketitle

\centerline{\large\bf Abstract }
\par
In this paper, we obtain the dispersive estimates and global well-posedness of the nonlinear magnetic Schr\"odinger equation in $\mathbb{R}^{3}$ with nonlinearity $|u|^{p-1}u$ with exponent $\frac{5}{3}< p<5$ when the initial value stays in a suitable space  $\Sigma_{s}$. By proving the resolvent estimates with weight functions and the \textquotedblleft almost equivalence\textquotedblright  between $(-\Delta_{A})^{\frac{s}{2}}$ and $(-\Delta)^{\frac{s}{2}}$, we obtain the Strichartz estimates of  $|J_{A}(t)|^{s}u:=e^{\frac{i|x|^{2}}{4t}}(-t^{2}\Delta_{A})^{\frac{s}{2}}e^{\frac{-i|x|^{2}}{4t}}u$, therefore the dispersive decay estimates are obtained.\\

\noindent{\bf Key words:} Nonlinear magnetic Schr\"{o}dinger equation; Dispersive decay; Resolvent estimates; Strichartz estimates.\\
\par
\noindent {\bf AMS Subject Classifications 2020:}\  35B40; 35Q41; 35Q55; 81Q20.

\section{Introduction}
In this paper, we consider the nonlinear magnetic Schr\"odinger equation
\begin{equation}\label{1.1}
\left\{
\begin{split}
&(i\partial_{t}+\Delta_{A})u+\rho|u|^{p-1}u=0, \ x\in \mathbb{R}^{3},\ t\geqslant1,\\
&u(1)=u_{0},
\end{split}
\right.
\end{equation}
where $\Delta_{A}:=(\nabla+iA)^{2}=(\nabla+iA)(\nabla+iA),\ \rho=\pm1,\ \frac{5}{3}< p<5$ and $A(x)=(A_1(x), A_2(x), A_3(x))\in C^{2}(\mathbb{R}^{3},\mathbb{R}^{3})$ is a real-valued vector function which satisfies the following hypothesis $ (\textbf H)$: for some small $0<\delta<1$, 
\begin{equation}\label{1.2}
|A(x)|+|\nabla A(x)|+|\nabla\nabla A(x)|\leqslant \delta\langle x \rangle^{-12}.
\end{equation}
Therefore, we claim that $\sigma(\Delta_{A})=(-\infty,0)$. In fact, for $\forall u\in \mathscr{S}$, we have
\begin{equation}\label{1.3}
((\nabla+iA)^{2}u,u)_{L^{2}}
=-\int_{\mathbb{R}^{3}}\sum\limits_{j=1}^{3}|(\partial_{x_{j}}+iA_{j})u|^{2}dx<0.
\end{equation}
Furthermore, from Lemma A.2 in \cite{KK}, we obtain that $0$ is neither an eigenvalue nor a resonance of the magnetic Schr\"odinger operator $\Delta_{A}$ and $\sigma_{sc}(\Delta_{A})=\emptyset$.\\

We denote by $\Sigma_{s}$ ($s>0$) the Hilbert space $H^{s}_{x}(\mathbb{R}^{3})\cap L^{2}(\mathbb{R}^{3},|x|^{s}dx)\cap\dot{H}_{x}^{s_{1}}(\mathbb{R}^{3},|x|^{s_{2}}dx)$ equipped with the norm
\begin{equation*}
\|u\|_{\Sigma_{s}}^{2}:=\|u\|_{H_{x}^{s}(\mathbb{R}^{3})}^{2}+\||x|^{s}u\|^{2}_{L^{2}_{x}(\mathbb{R}^{3})}+\||x|^{s_{2}}(-\Delta)^{\frac{s_{1}}{2}}u\|^{2}_{L^{2}_{x}(\mathbb{R}^{3})},
\end{equation*}
where $s=s_{1}+s_{2}$ and $0<s_{i}<s$, $i=1,2$.\\

Our main result is the following:
\begin{theorem}\label{theorem1.1}
Let $A(x)$ satisfy hypothesis $(\textbf H)$ and $\frac{5}{3}<p<5$. Then for $\forall\varepsilon\in(0,1)$, $\forall s_{0}\in(\frac{3}{2},\frac{5}{3})$, $u_{0}\in\Sigma_{s}$ ($s>s_{0}$), and $\|u_{0}\|_{\Sigma_{s}}$ is sufficiently small, there exists a constant $C>0$ such that the solution of equation \eqref{1.1} satisfies the following decay estimate
\begin{equation}\label{1.4}
\|u\|_{L^{\infty}_{x}(\mathbb{R}^{3})}\leq Ct^{(\varepsilon-\frac{3}{2})\cdot\frac{3}{2s_{0}}}\cdot\|u_{0}\|_{\Sigma_{s}},\  t\geqslant1.
\end{equation}
\end{theorem}

\begin{remark}
In Theorem \ref{theorem1.1}, letting $s_{0}\rightarrow\frac{3}{2}$, we obtain that $(-\frac{3}{2}+\varepsilon)\cdot\frac{3}{2s_{0}}\rightarrow-\frac{3}{2}+\varepsilon$, that is to say, we have \textquotedblleft almost optimal decay\textquotedblright.
\end{remark}

For the well-posedness of weak solutions of the nonlinear Schr\"odinger equation, it was obtained for the case $A\neq0$ in \cite{K1,K2,T1}, and for the case $A=0$ in  \cite{Bo,Na}. For the local well-posedness of strong solutions of the same equation, it was acquired in \cite{K2,T2} when $A=0$, and in \cite{NS} when $A\neq0$. In Appendix A of this paper, we obtain that the initial value problem of the nonlinear magnetic Schr\"odinger equation is globally well-posed in  $L^{\infty}_{t,x}((1,\infty)\times\mathbb{R}^{3})$.

For the free linear Schr\"{o}dinger equation, the sharp decay result is the dispersive estimate: $\|e^{-it\Delta}f\|_{L^{\infty}(\mathbb{R}^{n})}\leqslant C|t|^{-\frac{n}{2}}\|f\|_{L^{1}(\mathbb{R}^{n})}, \ t\in \mathbb{R}$. For the nonlinear Schr\"{o}dinger equation, the dispersive estimate depends on the range of $p$ of the nonlinear term $|u|^{p-1}u$. The case $1+\frac{4}{n}\leqslant p<1+\frac{4}{n-1}$ for $n\geqslant3$ was answered positively by Strauss \cite{St3,St2}, who proved that the zero solution is the only asymptotically free solution when $1<p\leqslant1+\frac{2}{n}$ for $n\geqslant2$ and when $1<p\leqslant2$ for $n=1$ \cite{St1}. Using the idea of Glassey \cite{Gl}, this result was extended to the case $1<p\leqslant3$ for $n=1$ by Barab \cite{Ba}. When  $1+\frac{2}{n}<p<1+\frac{4}{n}$ for $n\geqslant1$, this estimate was proved by Mckean and Shatah \cite{MS}.

For the linear Schr\"{o}dinger equation with a real-valued potential, if $0$ was neither an eigenvalue nor a resonance for the Schr\"{o}dinger operator $\Delta_{V}:=\Delta-V(x)$, Journ\'{e}, Soffer and Sogge \cite{JSS} showed that
\begin{equation}\label{1.5}
\|e^{-it\Delta_{V}}\mathscr{P}_{ac}f\|_{L^{\infty}}\leqslant C|t|^{-\frac{n}{2}}\|f\|_{L^{1}},\ \forall f\in\mathscr{S}(\mathbb{R}^{n}),\ n\geqslant3,
\end{equation}
where $\mathscr{P}_{ac}$ was the projection onto the absolute continuous spectrum of $\Delta_{V}$, and the potential $V(x)$ satisfied $|V(x)|\leqslant C\langle x \rangle^{-(n+4)-}$ and $\ \mathscr{F}(V)\in L^{1}$. For $n=3$, the dispersive estimate \eqref{1.5} was obtained in turn, when the potential $V(x)$ satisfied $|V(x)|\leqslant C(1+|x|)^{-3-}$ by Goldberg and Schlag \cite{GS}, when $V(x)$ satisfied $\int_{\mathbb{R}^{6}}\frac{|V(x)|\ |V(y)|}{|x-y|^{2}}dxdy<(4\pi)^{2}$ and $\mathop{sup}\limits_{x\in \mathbb{R}^{3}}\int_{\mathbb{R}^{3}}\frac{|V(y)|}{|x-y|}dy<4\pi$ by Rodnianski and Schlag \cite{RS1}, and when $V(x)$ satisfied $V(x)\in L^{p}\cap L^{q}$ ($p<\frac{3}{2}<q$) and $|V(x)|\leqslant C(1+|x|)^{-2-}$ by Goldberg \cite{Go2}. For $n\neq3$, this estimate was proved, see \cite{CCV,FY,Go1, Sc,Ya1,Ya2}.

For the one-dimension nonlinear Schr\"{o}dinger equation  with the potential $V(x)$ and nonlinearity term $|u|^{p-1}u$ ($p>3$), the dispersive result was obtained  when $V(x)$ is a real-valued Schwartz function and $\sigma(-\Delta_{V})=[0,\infty)$ in \cite{CGV}. The argument in \cite{CGV} was based on the distorted Fourier transforms (see \cite{H}) and the equivalence property  $\||J|^{s}u\|_{L^{2}}\approx\||J_{V}|^{s}u\|_{L^{2}}$ ($0\leqslant s<\frac{1}{2}$). The proof of this equivalence depended on the properties of the Jost solutions of $-\Delta_{V}u=\zeta^{2}u$ in $\mathbb{R}$, the transmission coefficient $T(\zeta)$ and the reflection coefficients $R_{\pm}(\zeta)$ of the scattering matrix \cite{DT}. When $n\geqslant3$, a decay result $\|u\|_{L^{2p}_{x}}\leqslant Ct^{-\gamma}$ (where $1+\frac{2}{n}<p\leqslant1+\frac{4}{n-2}$ and $\gamma<\frac{n}{2}(1-\frac{1}{p})$) was proved with $V(x)$ satisfied some conditions as in \cite{LZ}. In this paper, our situation is more complicated than in \cite{CGV} for the reason that we are short of an explicit expression of the scattering matrix in the high dimensional space. Another difficulty is that the operator $|J_{A}(t)|^{s}$ does not have as good properties as $|J_{V}(t)|^{s}$ in \cite{CGV,LZ}.

For the magnetic Schr\"{o}dinger equation, it is has been extensively studied by many mathematicians and physicists almost since its advent. For example, Erdo\v{g}an, Goldberg and Schlag \cite{EGS} proved the Strichartz and smoothing estimates for Schr\"{o}dinger operators with almost critical magnetic potentials in three and higher dimensions. D'Ancona, Fanelli, Vega and Visciglia \cite{DFVV} obtained the endpoint Strichartz  estimates for the magnetic Schr\"{o}dinger equation. There have other results in \cite{DF,GST}. The dispersive estimate $L^{1}\rightarrow L^{\infty}$  with the most optimal decay rate $t^{-\frac{n}{2}}$ for magnetic Schr\"{o}dinger operators is still an interesting problem and has been studied by many scholars. The short-time dispersive estimate of solutions was proved as follows (Theorem 4 in \cite{Ya1})
\begin{equation}\label{1.6}
\|e^{it(-\Delta_{A}+V(x))}f\|_{L^{\infty}}\lesssim|t|^{-\frac{n}{2}}\|f\|_{L^{1}},\ t\leqslant T\ll1,\ n\geqslant3.
\end{equation}
Komech and Kopylova gave the following decay estimate in weighted norms for $n=3$
\begin{equation}\label{1.7}
\|e^{it(-\Delta_{A}+V(x))}\mathscr{P}_{c}f\|_{L^{2}_{-\sigma}}\leqslant C\langle t \rangle^{-\frac{3}{2}}\|f\|_{L^{2}_{\sigma}},\ t\in \mathbb{R},
\end{equation}
where $\sigma>\frac{5}{2}$ and $\mathscr{P}_{c}$ was a projection onto the continuous spectrum of operator $-\Delta_{A}+V(x)$, see \cite{KK}. In this paper, we consider the nonlinear magnetic Schr\"{o}dinger equation and obtain the dispersive decay estimate.

$Outline\ of\ this\ paper:$

In Section 2, we construct an operator $|J_{A}(t)|^{s}:=M(t)(-t^{2}\Delta_{A})^{\frac{s}{2}}M(-t)$ from the same form as the free operator $|J(t)|^{s}:=M(t)(-t^{2}\Delta)^{\frac{s}{2}}M(-t)$, where $M(t)$ $:=e^{\frac{i|x|^{2}}{4t}}$ (see \cite{C} for a detailed introduction), and have $V(s)=s(-\Delta_{A})^{\frac{s}{2}}+[x\cdot \nabla,(-\Delta_{A})^{\frac{s}{2}}]-i[(-\Delta_{A})^{\frac{s}{2}},x\cdot A]$ by commuting between $|J_{A}|^{s}$ and $i\partial_{t}+\Delta_{A}$. Next, we obtain an integral expression of $V(s)$ (see Lemma \ref{lemma2.4}) by applying the integral expression of $(-\Delta_{A})^{\frac{s}{2}}$.

In the most significant part-Section 3, we mainly study the resolvent of the magnetic Schr\"{o}dinger operator $\Delta_{A}$ in order to obtain $L^{r}_{x}(\mathbb{R}^{3})$-estimates ($1\leqslant r\leqslant2$) of $V(s)$. We know that Cuenin and Kenig \cite{CK} proved the half-order derivative estimates and $L^{q'}(\mathbb{R}^{n})\rightarrow L^{q}(\mathbb{R}^{n})$ ($q\in[2,\frac{2n}{n-2}]$) estimates of the  magnetic resolvent in $n\geqslant3$. The difficulty lies in proving the first-order derivative estimates of the magnetic resolvent and acquiring the power of $\tau$ in the estimates of the resolvent $R_{A}:=(\tau-\Delta_{A})^{-1}$. To do this, on the one hand, applying the resolvent identity and the limiting absorption principle, we prove some estimates about the free resolvent with weight functions. On the other hand, we discuss the resolvent operator respectively on the \textquotedblleft low energy component\textquotedblright and \textquotedblleft high energy component\textquotedblright to obtain different powers of $\tau$. For the \textquotedblleft low energy component\textquotedblright, we show that if the potential $A(x)$ satisfies smallness condition $(\textbf H)$, then the power of $\tau$ can be reduced and some remainder terms from commutators can be well handled.

In Section 4, we prove Theorem \ref{theorem1.1}. The proof comprises two parts by applying the Strichartz estimates of the solution $|J_{A}(t)|^{s}u$ of the nonlinear Schr\"odinger equation \eqref{4.18}. On the one hand, the estimate of the term with respect to $V(s)$ has been dealt with in Section 3. On the other hand, we have to estimate the term $|J_{A}|^{s}(|u|^{p-1}u)$. For this purpose, we first prove an \textquotedblleft almost equivalence\textquotedblright between $(-\Delta)^{\frac{s}{2}}$ and $(-\Delta_{A})^{\frac{s}{2}}$, therefore the relationship between $|J|^{s}$ and $|J_{A}|^{s}$ is obtained. Next, applying the iterative method, we obtain the estimate of the term $|J_{A}|^{s}(|u|^{p-1}u)$. Combining the results of these two parts, we complete the Strichartz estimate of $|J_{A}(t)|^{s}u$. However, in order to obtain  decay result, we also need the following inequality (see Lemma \ref{lemma2.5})
\begin{equation*}
\|u\|_{L^{\infty}_{x}(\mathbb{R}^{3})}\leqslant C\|u\|_{L^{2}_{x}(\mathbb{R}^{3})}^{1-\frac{3}{2s}}\cdot\|u\|_{\dot{H}^{s}(\mathbb{R}^{3})}^{\frac{3}{2s}},\ \forall s>\frac{3}{2}.
\end{equation*}
Next, applying again the relationship between $|J|^{s}$ and $|J_{A}|^{s}$ and the Strichartz estimate of $|J_{A}(t)|^{s}u$, the dispersive decay estimate is obtained, that is to say, the proof of Theorem \ref{theorem1.1} is completed.

\bigskip
\hspace{-5mm}\textbf{Notation}:
\begin{itemize}
  \item $\mathscr{S}(\mathbb{R}^{n})$ is a Schwartz space, i.e. the set of all real- or complex-valued $C^{\infty}$ functions on $\mathbb{R}^{n}$ such that for ever nonnegative integer $m$ and every multi-index $\alpha$, $\sup\limits_{x\in\mathbb{R}^{n}}(1+|x|^{2})^{\frac{m}{2}}|D^{\alpha}u(x)|<\infty$.

  \item $\dot{H}^{s}(\mathbb{R}^{n})=\{u\in L^{2}(\mathbb{R}^{n}):|\xi|^{s}\cdot\widehat{u}\in L^{2}(\mathbb{R}^{n})\}.$
  \item
  $L^{2}_{\sigma}(\mathbb{R}^{n})=L^{2,\sigma}(\mathbb{R}^{n})=\{u(x):(1+|x|^{2})^{\frac{\sigma}{2}}u(x)\in L^{2}(\mathbb{R}^{n})\};$  $H^{m,\sigma}(\mathbb{R}^{n})=\{u(x):D^{\alpha}u\in L^{2}_{\sigma}(\mathbb{R}^{n}), 0\leqslant |\alpha|\leqslant m\}.$

  \item $\mathscr{F}u(x)=(2\pi)^{-\frac{n}{2}}\int_{\mathbb{R}^{n}}e^{-ix\cdot \xi}u(\xi)d\xi$, $(\mathscr{F}^{-1}u)(x)=(2\pi)^{-\frac{n}{2}}\int_{\mathbb{R}^{n}}e^{ix\cdot \xi}u(\xi)d\xi$.

  \item $[A,B]:=AB-BA$.

  \item $A\lesssim B$ means that there exists a constant $C>0$ such that $A\leqslant CB$.

\item $A\approx B$ means that there exists a constant $C>0$ such that $C^{-1}\leqslant \dfrac{B}{A}\leqslant C$.
\end{itemize}

\section{\bf Preliminaries}
In this section, first we introduce the dilation operator and the multiplier operator for every $t>0$
\begin{equation}\label{2.1}
D(t)f(x)=(2it)^{-\frac{3}{2}}f(\frac{x}{2t}),\end{equation}
\begin{equation}\label{2.2}
M(t)f(x)=e^{\frac{i|x|^{2}}{4t}}f(x),
\end{equation}
then
\begin{equation}\label{2.3}
e^{it\Delta}=M(t)D(t)\mathscr{F}^{-1}M(t).
\end{equation}
Furthermore, we have
\begin{equation}\label{2.4}
e^{it\Delta}x_{j}e^{-it\Delta}=M(t)2ti\partial x_{j}M(-t).
\end{equation}
Therefore, we introduce the operators
\begin{equation}\label{2.5}
J_{j}=M(t)2ti\partial x_{j}M(-t)=2tie^{\frac{i|x|^{2}}{4t}}\partial x_{j}e^{-\frac{i|x|^{2}}{4t}}=2ti\partial x_{j}+x_{j},
\end{equation}
and we have
$$[i\partial_{t}+\Delta,J_{j}]=0,$$
where $j=1,2,3$.\\
By direct calculation, it follows that
\begin{flalign}
&[i\partial_{t}+\Delta,M(t)]=M(t)(\frac{3i}{2t}+\frac{ix\cdot \nabla}{t}),&\nonumber\\ &[i\partial_{t}+\Delta,M(-t)]=M(-t)(-\frac{3i}{2t}-\frac{ix\cdot \nabla}{t}-\frac{|x|^{2}}{2t^{2}});&\label{2.6}
\end{flalign}
\begin{flalign}
&[i\nabla\cdot A,M(t)]=[iA\cdot\nabla,M(t)]=-\frac{x\cdot A}{2t}M(t),\ [|A|^{2},M(t)]=0,&\nonumber\\ &[i\partial_{t}+\Delta_{A},(-t^{2}\Delta_{A})^{\frac{s}{2}}]=\frac{is}{t}(-t^{2}\Delta_{A})^{\frac{s}{2}}.&\label{2.7}
\end{flalign}
Applying \eqref{2.6} and \eqref{2.7}, we obtain that
\begin{equation}\label{2.8}
[i\partial_{t}+\Delta_{A},M(t)]
=\frac{|x|^{2}}{4t^{2}}M(t)+M(t)(\frac{3i}{2t}-\frac{|x|^{2}}{4t^{2}}+\frac{ix\cdot \nabla}{t})-\frac{x\cdot A}{t}M(t).
\end{equation}
Similarly,
\begin{equation}\label{2.9}
[i\partial_{t}+\Delta_{A},M(-t)]=-\frac{|x|^{2}}{4t^{2}}M(-t)+M(-t)(-\frac{3i}{2t}-\frac{|x|^{2}}{4t^{2}}-\frac{ix\cdot \nabla}{t})+\frac{x\cdot A}{t}M(-t).
\end{equation}
Next we introduce the following two operators for any $s\geqslant0$
\begin{equation}\label{2.10}
|J(t)|^{s}:=M(t)(-t^{2}\Delta)^{\frac{s}{2}}M(-t),
\end{equation}
\begin{equation}\label{2.11} |J_{A}(t)|^{s}:=M(t)(-t^{2}\Delta_{A})^{\frac{s}{2}}M(-t).
\end{equation}
\begin{lemma}\label{lemma2.1}
For any $s\geqslant 0$ and $t>0$, we have
$$[i\partial_{t}+\Delta_{A},|J_{A}(t)|^{s}]=it^{s-1}M(t)V(s)M(-t),
$$
where \begin{equation}\label{2.12}
V(s)=s(-\Delta_{A})^{\frac{s}{2}}+[x\cdot \nabla,(-\Delta_{A})^{\frac{s}{2}}]-i[(-\Delta_{A})^{\frac{s}{2}},x\cdot A].
\end{equation}
\end{lemma}
\begin{proof}
Applying \eqref{2.8}-\eqref{2.11}, it follows that
\begin{equation*}
\begin{split}
&[i\partial_{t}+\Delta_{A},|J_{A}(t)|^{s}]\\
=&[i\partial_{t}+\Delta_{A},M(t)(-t^{2}\Delta_{A})^{\frac{s}{2}}M(-t)]\\
=&[i\partial_{t}+\Delta_{A},M(t)](-t^{2}\Delta_{A})^{\frac{s}{2}}M(-t)+M(t)[i\partial_{t}
+\Delta_{A},(-t^{2}\Delta_{A})^{\frac{s}{2}}M(-t)]\\
=&[i\partial_{t}+\Delta_{A},M(t)](-t^{2}\Delta_{A})^{\frac{s}{2}}M(-t)
+M(t)[i\partial_{t}+\Delta_{A},(-t^{2}\Delta_{A})^{\frac{s}{2}}]M(-t)\\
&+M(t)(-t^{2}\Delta_{A})^{\frac{s}{2}}
[i\partial_{t}+\Delta_{A},M(-t)]\\
=&(\frac{|x|^{2}}{4t^{2}}M(t)+M(t)(\frac{3i}{2t}-\frac{|x|^{2}}{4t^{2}}+\frac{ix\cdot \nabla}{t})-\frac{x\cdot A}{t}M(t))(-t^{2}\Delta_{A})^{\frac{s}{2}}M(-t)
+\frac{is}{t}|J_{A}(t)|^{s}\\
&+M(t)(-t^{2}\Delta_{A})^{\frac{s}{2}}(-\frac{|x|^{2}}{4t^{2}}M(-t)+M(-t)(-\frac{3i}{2t}
-\frac{|x|^{2}}{4t^{2}}-\frac{ix\cdot \nabla}{t})+\frac{x\cdot A}{t}M(-t))\\
\end{split}
\end{equation*}
\begin{equation*}
\begin{split}
=&\frac{3i}{2t}|J_{A}(t)|^{s}+\frac{i}{t}M(t)x\cdot \nabla(-t^{2}\Delta_{A})^{\frac{s}{2}}M(-t)-\frac{x\cdot A}{t}M(t)(-t^{2}\Delta_{A})^{\frac{s}{2}}M(-t)\\
&+\frac{is}{t}|J_{A}(t)|^{s}-\frac{3i}{2t}|J_{A}(t)|^{s}-\frac{i}{t}M(t)(-t^{2}\Delta_{A})^{\frac{s}{2}}M(-t)x\cdot\nabla\\
&-M(t)(-t^{2}\Delta_{A})^{\frac{s}{2}}M(-t)\frac{|x|^{2}}{2t^{2}}+M(t)(-t^{2}\Delta_{A})^{\frac{s}{2}}\frac{x\cdot A}{t}M(-t)\\
\end{split}
\end{equation*}
\begin{equation*}
\begin{split}
=&\frac{is}{t}|J_{A}(t)|^{s}+\frac{i}{t}M(t)[x\cdot \nabla,(-t^{2}\Delta_{A})^{\frac{s}{2}}M(-t)]-M(t)(-t^{2}\Delta_{A})^{\frac{s}{2}}M(-t)\frac{|x|^{2}}{2t^{2}}\\
&+M(t)[(-t^{2}\Delta_{A})^{\frac{s}{2}},\frac{x\cdot A}{t}]M(-t)\\=&\frac{is}{t}|J_{A}(t)|^{s}+\frac{i}{t}M(t)([x\cdot \nabla,(-t^{2}\Delta_{A})^{\frac{s}{2}}]M(-t)+(-t^{2}\Delta_{A})^{\frac{s}{2}}[x\cdot \nabla,M(-t)])\\
&-M(t)(-t^{2}\Delta_{A})^{\frac{s}{2}}M(-t)\frac{|x|^{2}}{2t^{2}}
+M(t)[(-t^{2}\Delta_{A})^{\frac{s}{2}},\frac{x\cdot A}{t}]M(-t)\\
=&\frac{is}{t}|J_{A}(t)|^{s}+\frac{i}{t}M(t)[x\cdot \nabla,(-t^{2}\Delta_{A})^{\frac{s}{2}}]M(-t)+M(t)[(-t^{2}\Delta_{A})^{\frac{s}{2}},\frac{x\cdot A}{t}]M(-t)\\
=&it^{s-1}(sM(t)(-\Delta_{A})^{\frac{s}{2}}M(-t)+M(t)[x\cdot \nabla,(-\Delta_{A})^{\frac{s}{2}}]M(-t)\\& -iM(t)[(-\Delta_{A})^{\frac{s}{2}},x\cdot A]M(-t))\\
=&it^{s-1}M(t)V(s)M(-t).
\end{split}
\end{equation*}
This completes the proof of Lemma \ref{lemma2.1}.
\end{proof}
For operator $ V(s)$, we will introduce it in detail by showing an integral representation formula. To do this, we need the following lemmas:
\begin{lemma}\label{lemma2.2}
Let $A(x)$ satisfy hypothesis $(\textbf H)$, and
\begin{equation}\label{2.13}
V_{1}:=(2+x\cdot\nabla)|A|^{2}-2i\nabla\cdot A-2ix\cdot(\nabla A)\cdot\nabla-ix\cdot(\nabla\nabla A),
\end{equation}
where $x\cdot(\nabla A)\cdot\nabla u=\sum\limits_{k,j=1}^{3}(x_{k}(\partial_{x_{k}}A_{k})\partial_{x_{k}}u+x_{k}(\partial_{x_{k}}A_{j})\partial_{x_{j}}u)$, and $x\cdot(\nabla\nabla A)u=\sum\limits_{k,j=1}^{3}(x_{k}(\partial^{2}_{x_{k}x_{k}}A_{k})u+x_{k}(\partial^{2}_{x_{k}x_{j}}A_{j})u)$. Then for any $\tau>0$, we have
$$
[x\cdot\nabla,(-\Delta_{A})(\tau-\Delta_{A})^{-1}]=2\tau\Delta_{A}(\tau-\Delta_{A})^{-2}+\tau(\tau-\Delta_{A})^{-1}V_{1}(\tau-\Delta_{A})^{-1}.
$$
\end{lemma}

\begin{proof}
Let us calculate that
\begin{equation*}
\begin{split}
&[x\cdot\nabla,-\Delta_{A}]\\
=&[x\cdot\nabla,-\Delta-iA\cdot\nabla-i\nabla\cdot A+|A|^{2}]\\
=&[x\cdot\nabla,-\Delta]-i[x\cdot\nabla,A\cdot\nabla]-i[x\cdot\nabla,\nabla\cdot A]+[x\cdot\nabla,|A|^{2}]\\
=&2\Delta-2ix\cdot(\nabla  A)\cdot\nabla+2iA\cdot\nabla-ix\cdot(\nabla\nabla A)+x\cdot(\nabla|A|^{2})\\
=&2\Delta_{A}-2i\nabla\cdot A+(2+x\cdot\nabla)|A|^{2}-2ix\cdot(\nabla  A)\cdot\nabla-ix\cdot(\nabla\nabla A),
\end{split}
\end{equation*}
therefore, according to the notation of $V_{1}$, it follows that
\begin{equation}\label{2.14}
[x\cdot\nabla,-\Delta_{A}]=2\Delta_{A}+V_{1}.
\end{equation}
Consequently, we obtain from \eqref{2.14} that
\begin{flalign}
&[x\cdot\nabla,(-\Delta_{A})(\tau-\Delta_{A})^{-1}]\nonumber\\
=&[x\cdot\nabla,-\Delta_{A}](\tau-\Delta_{A})^{-1}-\Delta_{A}[x\cdot\nabla,(\tau-\Delta_{A})^{-1}]\nonumber\\
=&[x\cdot\nabla,-\Delta_{A}](\tau-\Delta_{A})^{-1}+\Delta_{A}(\tau-\Delta_{A})^{-1}[x\cdot\nabla,-\Delta_{A}](\tau-\Delta_{A})^{-1}\nonumber\\
=&2\Delta_{A}(\tau-\Delta_{A})^{-1}+V_{1}(\tau-\Delta_{A})^{-1}+2\Delta_{A}^{2}(\tau-\Delta_{A})^{-2}\nonumber\\
&+\Delta_{A}(\tau-\Delta_{A})^{-1}V_{1}(\tau-\Delta_{A})^{-1}\nonumber\\
=&2\tau\Delta_{A}(\tau-\Delta_{A})^{-2}+\tau(\tau-\Delta_{A})^{-1}V_{1}(\tau-\Delta_{A})^{-1},\nonumber
\end{flalign}
which completes the proof of Lemma \ref{lemma2.2}.
\end{proof}

\begin{lemma}\label{lemma2.3}
Let $A(x)$ satisfy hypothesis $(\textbf H)$, $Q:=x\cdot A$ and \begin{equation}\label{2.15}
V_{2}:=2A\cdot\nabla+2x\cdot(\nabla A)\cdot\nabla+2(div A)+x\cdot(\nabla\nabla A)-|A|^{2}-x\cdot(\nabla A)\cdot A,
\end{equation}
where $x\cdot(\nabla A)\cdot Au=\sum\limits_{k,j=1}^{3}(x_{k}A_{k}(\partial_{x_{k}}A_{k})u+x_{k}A_{j}(\partial_{x_{j}}A_{k})u)$. Then for any $\tau>0$, we have
$$
[x\cdot A,(-\Delta_{A})(\tau-\Delta_{A})^{-1}]=\tau(\tau-\Delta_{A})^{-1}V_{2}(\tau-\Delta_{A})^{-1}.
$$
\end{lemma}

\begin{proof}
It is easy to calculate that
\begin{equation*}
\begin{split}
&[x\cdot A,-\Delta]=2A\cdot\nabla+2x\cdot(\nabla A)\cdot\nabla+2(div A)+x\cdot(\nabla\nabla A);\\
&[x\cdot A,A\cdot\nabla]=-|A|^{2}-x\cdot(\nabla A)\cdot A;\\
&[x\cdot A,divA]=[Q,|A|^{2}]=0.
\end{split}
\end{equation*}
Using the commutations above, we deduce that
\begin{flalign}
&[x\cdot A,-\Delta_{A}]\nonumber\\
=&[x\cdot A,-\Delta]-2i[x\cdot A,A\cdot\nabla]+i[x\cdot A,divA]+[x\cdot A,|A|^{2}]\nonumber\\
=&2A\cdot\nabla+2x\cdot(\nabla A)\cdot\nabla+2(div A)+x\cdot(\nabla\nabla A)-|A|^{2}-x\cdot(\nabla A)\cdot A.\label{2.16}
\end{flalign}
Consequently, according to \eqref{2.16} and the notation of $V_{2}$, we have
\begin{flalign}
&[x\cdot A,(-\Delta_{A})(\tau-\Delta_{A})^{-1}]\nonumber\\
=&[x\cdot A,-\Delta_{A}](\tau-\Delta_{A})^{-1}
+\Delta_{A}(\tau-\Delta_{A})^{-1}
[x\cdot A,-\Delta_{A}](\tau-\Delta_{A})^{-1}\nonumber\\
=&V_{2}(\tau-\Delta_{A})^{-1}+\Delta_{A}(\tau-\Delta_{A})^{-1}V_{2}
(\tau-\Delta_{A})^{-1}\nonumber\\
=&(I+\Delta_{A}(\tau-\Delta_{A})^{-1})V_{2}(\tau-\Delta_{A})^{-1}\nonumber\\
=&\tau(\tau-\Delta_{A})^{-1}V_{2}(\tau-\Delta_{A})^{-1}.\nonumber
\end{flalign}
This completes the proof of Lemma \ref{lemma2.3}.
\end{proof}

\begin{lemma}\label{lemma2.4}
Let $A(x)$ satisfy hypothesis $(\textbf H)$, $V_{1}$ and $V_{2}$ be the operators in \eqref{2.13} and \eqref{2.15} respectively, and $V(s)$ be the operator in \eqref{2.12}. Then for $0<s<2$ and $\tau>0$, we have the integral representation of $V(s)$ with a constant $c(s)$
\begin{flalign}
V(s)=&c(s)\int^{\infty}_{0}\tau^{\frac{s}{2}}(\tau-\Delta_{A})^{-1}V_{1}(\tau-\Delta_{A})^{-1}d\tau\nonumber\\
+&ic(s)\int^{\infty}_{0}\tau^{\frac{s}{2}}(\tau-\Delta_{A})^{-1}V_{2}(\tau-\Delta_{A})^{-1}d\tau.\label{2.17}
\end{flalign}
\end{lemma}

\begin{proof}
Let us recall the formula	
\begin{equation}\label{2.18}
(-\Delta_{A})^{\frac{s}{2}}=c(s)(-\Delta_{A})\int^{\infty}_{0}\tau^{\frac{s}{2}-1}(\tau-\Delta_{A})^{-1}d\tau,
\end{equation}
where $(c(s))^{-1}=\int^{\infty}_{0}\tau^{\frac{s}{2}-1}(\tau+1)^{-1}d\tau$ with $0<s<2$ (also can see \cite{I}). \\
Inserting \eqref{2.18} into \eqref{2.12}, we obtain that
\begin{equation*}
\begin{split}
V(s)&=s(-\Delta_{A})^{\frac{s}{2}}+c(s)\int^{\infty}_{0}\tau^{\frac{s}{2}-1}[x\cdot\nabla,(-\Delta_{A})(\tau-\Delta_{A})^{-1}]d\tau\\
&+ic(s)\int^{\infty}_{0}\tau^{\frac{s}{2}-1}[x\cdot A,(-\Delta_{A})(\tau-\Delta_{A})^{-1}]d\tau.
\end{split}
\end{equation*}
From Lemma \ref{lemma2.2} and Lemma \ref{lemma2.3}, it follows that
\begin{flalign}
V(s)&=s(-\Delta_{A})^{\frac{s}{2}}+2c(s)\int^{\infty}_{0}\tau^{\frac{s}{2}}\Delta_{A}(\tau-\Delta_{A})^{-2}d\tau\nonumber\\
&+c(s)\int^{\infty}_{0}\tau^{\frac{s}{2}}(\tau-\Delta_{A})^{-1}V_{1}(\tau-\Delta_{A})^{-1}d\tau\nonumber\\
&+ic(s)\int^{\infty}_{0}\tau^{\frac{s}{2}}(\tau-\Delta_{A})^{-1}V_{2}(\tau-\Delta_{A})^{-1}d\tau.\label{2.19}
\end{flalign}
For the second item of the right-hand side of \eqref{2.19}, we have
\begin{equation*}
\begin{split}
&2c(s)\int^{\infty}_{0}\tau^{\frac{s}{2}}\Delta_{A}(\tau-\Delta_{A})^{-2}d\tau\\
=&-2c(s)\int^{\infty}_{0}\int^{\infty}_{0} \tau^{\frac{s}{2}}\lambda(\tau+\lambda)^{-2}dE_{\lambda}d\tau\\
=&-2c(s)\int^{\infty}_{0}\lambda\cdot\frac{s}{2}\int^{\infty}_{0}\tau^{\frac{s}{2}-1}(\tau+\lambda)^{-1}d\tau dE_{\lambda}\\
=&2c(s)\Delta_{A}\cdot\frac{s}{2}
\int^{\infty}_{0}\tau^{\frac{s}{2}-1}(\tau-\Delta_{A})^{-1}d\tau\\
=&sc(s)\Delta_{A}\int^{\infty}_{0}\tau^{\frac{s}{2}-1}(\tau-\Delta_{A})^{-1}d\tau\\
=&-s(-\Delta_{A})^{\frac{s}{2}}.
\end{split}
\end{equation*}
Hence,
\begin{equation*}
\begin{split}
V(s)=&s(-\Delta_{A})^{\frac{s}{2}}-s(-\Delta_{A})^{\frac{s}{2}}+c(s)\int^{\infty}_{0}\tau^{\frac{s}{2}}
(\tau-\Delta_{A})^{-1}V_{1}(\tau-\Delta_{A})^{-1}d\tau\\
&+ic(s)\int^{\infty}_{0}\tau^{\frac{s}{2}}(\tau-\Delta_{A})^{-1}V_{2}(\tau-\Delta_{A})^{-1}d\tau\\
=&c(s)\int^{\infty}_{0}\tau^{\frac{s}{2}}(\tau-\Delta_{A})^{-1}V_{1}(\tau-\Delta_{A})^{-1}d\tau\\
&+ic(s)\int^{\infty}_{0}\tau^{\frac{s}{2}}(\tau-\Delta_{A})^{-1}V_{2}(\tau-\Delta_{A})^{-1}d\tau.
\end{split}
\end{equation*}
This completes the proof of Lemma \ref{lemma2.4}.
\end{proof}
In order to obtain the decay result in Section 4, we need the following lemma:
\begin{lemma}\label{lemma2.5}
For any $s>\frac{3}{2}$, and $u\in C_{0}^{\infty}(\mathbb{R}^{3})$, there exists a fixed $C>0$ such that
\begin{equation}\label{2.20}
\|u\|_{L^{\infty}_{x}(\mathbb{R}^{3})}\leqslant C\|u\|_{L^{2}_{x}(\mathbb{R}^{3})}^{1-\frac{3}{2s}}\cdot\|u\|_{\dot{H}^{s}(\mathbb{R}^{3})}^{\frac{3}{2s}}.
\end{equation}
\end{lemma}
\begin{proof}
From properties of the Fourier transform, it follows that
\begin{equation}\label{2.21}
\|u\|_{L_{x}^{\infty}}\leqslant C\|\mathscr{F}u\|_{L_{x}^1}.
\end{equation}
Furthermore, for $\forall\ k>0$
\begin{flalign}
\|\mathscr{F}u\|_{L_{x}^{1}(\mathbb{R}^{3})} &\leqslant \|\mathscr{F}u\|_{L_{x}^{1}(|\xi|\leqslant k)}+\|\mathscr{F}u\|_{L_{x}^{1}(|\xi|\geqslant k)}\nonumber\\
& \leqslant \|\mathscr{F}u\|_{L_{x}^{2}(|\xi|\leqslant k)}(\int_{|\xi|\leqslant k}1d\xi)^{\frac{1}{2}}+\int_{|\xi|\geqslant k}(\mathscr{F}u)\cdot|\xi|^{s}\cdot|\xi|^{-s}d\xi\nonumber\\
& \leqslant \|\mathscr{F}u\|_{L_{x}^{2}(|\xi|\leqslant k)}\cdot\sqrt{\frac{4\pi}{3}}k^{\frac{3}{2}}+\|(\mathscr{F}u)\cdot|\xi|^{s}\|_{L_{x}^{2}(|\xi|\geqslant k)}\||\xi|^{-s}\|_{L_{x}^{2}(|\xi|\geqslant k)}\nonumber\\
& \leqslant \|\mathscr{F}u\|_{L_{x}^{2}(\mathbb{R}^{3})}\cdot\sqrt{\frac{4\pi}{3}}k^{\frac{3}{2}}+\sqrt{\frac{4\pi}{2s-3}}k^{\frac{3}{2}-s}\|u\|_{\dot{H}^{s}(\mathbb{R}^{3})}, \label{2.22}
\end{flalign}
let $k=(\sqrt{\frac{3}{2s-3}}\|u\|_{\dot{H}_{x}^{s}})^{\frac{1}{s}}\cdot\|u\|^{-\frac{1}{s}}_{L_{x}^{2}}$, then the last two terms of \eqref{2.22} are equal.
Combining \eqref{2.21} and \eqref{2.22}, inequality \eqref{2.20} follows. This completes the proof of Lemma \ref{lemma2.5}.
\end{proof}

\section{\bf Resolvent estimates}
In this section, according to the integral representation of $V(s)$ in Lemma \ref{lemma2.4}, we need to estimate the terms of resolvent with weight functions before we obtain the estimate of $V(s)$.
\begin{lemma}\label{lemma3.1}
Let $R_{0}=(\tau-\Delta)^{-1}$, we have the weighted resolvent estimates for $\tau>0$ and $N>3$
\begin{flalign}
&\|(\tau-\Delta)^{-1}f\|_{L^{2}}\leqslant C\tau^{-1}\|f\|_{L^{2}},\label{3.1}\\
&\|(\tau-\Delta)^{-1}\langle x\rangle ^{-N}f\|_{L^{2}}\leqslant C\tau^{-\frac{1}{2}}\|f\|_{L^{2}},\label{3.2}\\
&\|\langle x\rangle ^{-N}(\tau-\Delta)^{-1}f\|_{L^{2}}\leqslant C\tau^{-\frac{1}{2}}\|f\|_{L^{2}}.\label{3.3}
\end{flalign}
\end{lemma}

\begin{proof}
Let us recall the formula of the free resolvent operator $R_{0}=(\tau-\Delta)^{-1}$
\begin{flalign}
R_{0}f(x)=&(\tau-\Delta)^{-1}f(x)\nonumber\\
=&\mathscr{F}^{-1}(\frac{1}{|\xi|^{2}+\tau}\mathscr{F}f)\nonumber\\
=&(\mathscr{F}^{-1}(\frac{1}{|\xi|^{2}+\tau}))\ast f\nonumber\\
:=&\int_{\mathbb{R}^{3}}K(\tau,x-y)f(y)dy.\label{3.4}
\end{flalign}
Next, we will calculate the integral kernel $K(\tau,x-y)$ of the free resolvent operator	
\begin{flalign}	
K(\tau,x-y)=&(2\pi)^{-\frac{3}{2}}\int_{\mathbb{R}^{3}}\frac{e^{i(x-y)\cdot\xi}}{|\xi|^{2}+\tau}d\xi\nonumber\\
=&(2\pi)^{-\frac{3}{2}}\int_{0}^{\infty}\frac{r^{2}}{r^{2}+\tau}\int_{|\omega|=1}e^{i(x-y)\cdot r\omega}d\omega dr\nonumber\\
=&(2\pi)^{-\frac{3}{2}}\int_{0}^{\infty}\frac{r^{2}}{r^{2}+\tau}\cdot(r|x-y|)^{-\frac{1}{2}}\cdot J_{\frac{1}{2}}(r|x-y|)dr\nonumber\\
=&\frac{1}{2\pi^{2}}\cdot\frac{1}{|x-y|}\int_{0}^{\infty}\frac{r}{r^{2}+\tau}sin(r|x-y|)dr\nonumber\\
=&\frac{1}{2\pi^{2}}\cdot\frac{1}{|x-y|}\int_{0}^{\infty}\frac{r}{r^{2}+\tau}sin(r|x-y|)dr\nonumber\\
=&\frac{1}{4\pi^{2}i}\cdot\frac{1}{|x-y|}\int_{-\infty}^{\infty}\frac{r}{r^{2}+\tau}e^{i|x-y|r}dr\nonumber\\
=&\frac{1}{2\pi}\cdot\frac{1}{|x-y|}Res\{\frac{r}{r^{2}+\tau}e^{i|x-y|r}, i\tau^{\frac{1}{2}}\}\nonumber\\
=&\frac{1}{2\pi}\cdot\frac{e^{-|x-y| \tau^{\frac{1}{2}}}}{|x-y|}\nonumber\\
=&\frac{1}{2\pi}\cdot\frac{e^{-|x-y|\tau^{\frac{1}{2}}}}{|x-y|},\label{3.5}
\end{flalign}	
where $J_{\frac{1}{2}}(x)=(\frac{2}{\pi x})^{\frac{1}{2}}\cdot \sin x$ is a Bessel function.\\
From Young's inequality and the expression of the integral kernel $K(\tau,x-y)$ in \eqref{3.5}, we obtain that
\begin{flalign}	
\|(\tau-\Delta)^{-1}f\|_{L^{2}}\leqslant &\|K(\tau,x)\|_{L^{1}}\|f\|_{L^{2}}\nonumber\\
=&\frac{1}{2\pi}\int_{\mathbb{R}^{3}}\frac{e^{-|x-y|\tau^{\frac{1}{2}}}}{|x-y|}dx\cdot\|f\|_{L^{2}}\nonumber\\
=&2\pi\int_{0}^{\infty}re^{-r\tau^{\frac{1}{2}}}dr\cdot\|f\|_{L^{2}}\nonumber\\
\leqslant&C\tau^{-1}\|f\|_{L^{2}},\label{3.6}
\end{flalign}
which completes the proof of estimate \eqref{3.1}.\\
Now we turn to \eqref{3.2}, again from Young's inequality, identity \eqref{3.5} and H\"{o}lder's inequality, we have
\begin{flalign}	
\|(\tau-\Delta)^{-1}\langle x\rangle ^{-N}f\|_{L^{2}}&\leqslant C\|K(\tau,x)\|_{L^{\frac{3}{2}}}\|\langle x\rangle ^{-N}f\|_{L^{\frac{6}{5}}}\nonumber\\
&\leqslant C\|K(\tau,x)\|_{L^{\frac{3}{2}}}\|f\|_{L^{2}}\cdot\|\langle x\rangle ^{-N}\|_{L^{3}},\label{3.7}
\end{flalign}
and we get
\begin{flalign}
\|K(\tau,x)\|_{L^{\frac{3}{2}}}=&\|\frac{1}{2\pi}\cdot\frac{e^{-|x|\tau^{\frac{1}{2}}}}{|x|}\|_{L^{\frac{3}{2}}}\nonumber\\
\mathop{=}\limits^{|x|=r}&(2\int_{0}^{\infty}\frac{e^{-\frac{3}{2}\tau^{\frac{1}{2}}r}}{r^{\frac{3}{2}}}\cdot r^{2}dr)^{\frac{2}{3}}\nonumber\\	\mathop{=}\limits^{r^{\frac{1}{2}}=\eta}&(4\int_{0}^{\infty}e^{-\frac{3}{2}\tau^{\frac{1}{2}}\eta^{2}}\cdot\eta^{2}d\eta)^{\frac{2}{3}}\nonumber\\
=&(\frac{4}{3}\tau^{-\frac{1}{2}}\int_{0}^{\infty}e^{-\frac{3}{2}\tau^{\frac{1}{2}}\eta^{2}}d\eta)^{\frac{2}{3}}\nonumber\\
=&(\frac{4}{3}\tau^{-\frac{1}{2}}\cdot\frac{1}{2}(\frac{3}{2})^{-\frac{1}{2}}\tau^{-\frac{1}{4}}\int_{\mathbb{R}}e^{-((\frac{3}{2})^{\frac{1}{2}}\tau^{\frac{1}{4}}\eta)^{2}}d((\frac{3}{2})^{\frac{1}{2}}\tau^{\frac{1}{4}}\eta))^{\frac{2}{3}}\nonumber\\
=&(\frac{2}{3}\tau^{-\frac{1}{2}}\cdot(\frac{2}{3})^{\frac{1}{2}}\tau^{-\frac{1}{4}}\cdot\pi^{\frac{1}{2}})^{\frac{2}{3}}\nonumber\\
=&\frac{2\pi^{\frac{1}{3}}}{3}\tau^{-\frac{1}{2}},\label{3.8}
\end{flalign}
since $\int_{\mathbb{R}}e^{-t^{2}}dt=\sqrt{\pi}$.\\
Therefore, applying \eqref{3.8} to \eqref{3.7}, we have $$\|(\tau-\Delta)^{-1}\langle x\rangle ^{-N}f\|_{L^{2}}\leqslant C\tau^{-\frac{1}{2}}\|f\|_{L^{2}},$$
which completes the proof of estimate \eqref{3.2}.\\
Similarly, according to H\"{o}lder's inequality, identity \eqref{3.8} and Young's inequality, we have
\begin{flalign}	
&\|\langle x\rangle ^{-N}(\tau-\Delta)^{-1}f\|_{L^{2}}\nonumber\\
=&\|\langle x\rangle ^{-N}(K\ast f)\|_{L^{2}}\nonumber\\
\leqslant&\|\langle x\rangle ^{-N}\|_{L^{3}}\cdot\|K\ast f\|_{L^{6}}\nonumber\\
\leqslant&C\|K(\tau,x)\|_{L^{\frac{3}{2}}}\cdot\|f\|_{L^{2}}\nonumber\\
\leqslant&C\tau^{-\frac{1}{2}}\|f\|_{L^{2}},\label{3.9}
\end{flalign}
which completes the proof of estimate \eqref{3.3}.
\end{proof}

Similar the proof of Lemma \ref{lemma3.1}, we have	
\begin{corollary}\label{corollary3.1}
Let $R_{0}=(\tau-\Delta)^{-1}$, we have the weighted resolvent estimates for $\tau>0$ and $N>3$
\begin{flalign*}
&\|(\tau-\Delta)^{-1}f\|_{L^{r}}\leqslant C\tau^{-1}\|f\|_{L^{r}},\\
&\|(\tau-\Delta)^{-1}\langle x\rangle ^{-N}f\|_{L^{r}}\leqslant C\tau^{-\frac{1}{r}}\|f\|_{L^{r}},\\
&\|\langle x\rangle ^{-N}(\tau-\Delta)^{-1}f\|_{L^{r}}\leqslant C\tau^{-\frac{1}{r'}}\|f\|_{L^{r}},
\end{flalign*}	
where $1\leqslant r<\infty$	and $\frac{1}{r}+\frac{1}{r'}=1$.
\end{corollary}
		
\begin{lemma}\label{lemma3.2}
(the estimate of $V(s)$)\\
Let $V(s)$ be the operator that appears in  Lemma \ref{lemma2.4}, for $\forall\ 1\leqslant r\leqslant2$ and $\frac{2}{r}-1<s<2$, we have
\begin{equation}\label{3.10}
\|V(s)f\|_{L_{x}^{r}}\lesssim \|f\|_{\dot{H}_{x}^{1}},\ \forall f\in \mathscr{S}(\mathbb{R}^{3}).
\end{equation}
\end{lemma}

\begin{proof}
Let us recall Lemma \ref{lemma2.4}
\begin{equation*}
\begin{split}
V(s)=&c(s)\int^{\infty}_{0}\tau^{\frac{s}{2}}(\tau-\Delta_{A})^{-1}V_{1}(\tau-\Delta_{A})^{-1}d\tau\\
+&ic(s)\int^{\infty}_{0}\tau^{\frac{s}{2}}(\tau-\Delta_{A})^{-1}V_{2}(\tau-\Delta_{A})^{-1}d\tau,
\end{split}
\end{equation*}
where
\begin{equation*}
\begin{split}
&V_{1}:=(2+x\cdot\nabla)|A|^{2}-2i\nabla\cdot A-2ix\cdot(\nabla A)\cdot\nabla -ix\cdot(\nabla\nabla A),\\
&V_{2}:=2A\cdot\nabla+2x\cdot(\nabla A)\cdot\nabla+2(div A)+x\cdot(\nabla\nabla A)-|A|^{2}-x\cdot(\nabla A)\cdot A.
\end{split}
\end{equation*}
Therefore,
\begin{equation*}
\begin{split}	
\|V(s)f\|_{L_{x}^{r}}&\leqslant c(s)\int^{\infty}_{0}\tau^{\frac{s}{2}}\|(\tau-\Delta_{A})^{-1}V_{1}(\tau-\Delta_{A})^{-1}f\|_{L_{x}^{r}}d\tau\\
&+c(s)\int^{\infty}_{0}\tau^{\frac{s}{2}}\|(\tau-\Delta_{A})^{-1}V_{2}(\tau-\Delta_{A})^{-1}f\|_{L_{x}^{r}}d\tau.
\end{split}
\end{equation*}
Since $V_{1}$ and $V_{2}$ have similar items, we can estimate them in the same way. Next we will estimate $\int^{\infty}_{0}\tau^{\frac{s}{2}}\|(\tau-\Delta_{A})^{-1}V_{1}(\tau-\Delta_{A})^{-1}f\|_{L_{x}^{r}}d\tau$ and separate it into two parts:
\begin{flalign}
V_{1}:&=(2+x\cdot\nabla)|A|^{2}-2i\nabla\cdot A-2ix\cdot(\nabla A)\cdot\nabla -ix\cdot(\nabla\nabla A)\nonumber\\
&=((2+x\cdot\nabla)|A|^{2}-ix\cdot(\nabla\nabla A))-(2i\nabla\cdot A+2ix\cdot\nabla A\cdot\nabla)\nonumber\\
&:=\tilde{A}-\tilde{\tilde A}\cdot\nabla.\label{3.11}
\end{flalign}
Furthermore, we continue to separate these two parts into two parts: $\int^{\infty}_{0}=\int^{1}_{0}+\int^{\infty}_{1}$ respectively.\\
When $\tau\in(0,1)$, on the one hand, we deduce that
\begin{equation*}
\begin{split}
&\|(\tau-\Delta_{A})^{-1}\tilde{\tilde A}\cdot\nabla(\tau-\Delta_{A})^{-1}f\|_{L_{x}^{r}}\\
\lesssim &\|(\tau-\Delta_{A})^{-1}\langle x \rangle^{-6}\|_{L^{2}\rightarrow L^{r}}
\cdot\|\langle x \rangle^{6}\tilde{\tilde A}\cdot\nabla(\tau-\Delta_{A})^{-1}f\|_{L^{2}}.
\end{split}
\end{equation*}
$\ \ \ \ \ \ \ \ \ \ \ \ \ $:=\uppercase\expandafter{\romannumeral1}$\cdot$\uppercase\expandafter{\romannumeral2},\\
$Setp\ 1.\ Estimate$ \uppercase\expandafter{\romannumeral1}.\\
According to the resolvent identity, we obtain that
\begin{equation*}
\begin{split}
&\|(\tau-\Delta_{A})^{-1}\langle x \rangle^{-6}f\|_{L^{r}}\\
=&\|R_{0}(I+(|A|^{2}-iA\cdot\nabla-i\nabla\cdot A)R_{0})^{-1}\langle x\rangle^{-6}f\|_{L^{r}}\\
\lesssim&\|R_{0}\langle x\rangle^{-6}\|_{L^{2}\rightarrow L^{r}}\cdot\|\langle x\rangle^{6}(I+A\cdot\nabla R_{0})^{-1}\langle x\rangle^{-6}\|_{L^{2}\rightarrow L^{2}}\cdot\|f\|_{L^{2}},
\end{split}
\end{equation*}
where $R_{0}:=(\tau-\Delta)^{-1}$.\\
Applying the limiting absorption principle (see Theorem 4.1 in \cite{Ag}), we have
\begin{equation}\label{3.12}
R_{0}:L^{2,6}\rightarrow H^{2,-6},\ {\rm for} \  \tau>0.
\end{equation}
Computing directly, we obtain that
\begin{flalign}
\|Af\|_{L^{2,6}}&\leqslant\delta(\int_{R^{3}}\langle x\rangle^{12}\langle x\rangle^{-12\cdot2}|f|^{2}dx)^{\frac{1}{2}}\nonumber\\
&\leqslant\delta\|\langle x\rangle^{12-12}\|_{L^{\infty}}^{\frac{1}{2}}\cdot\|f\|_{L^{2,-6}}\nonumber\\
&\leqslant\delta\|f\|_{L^{2,-6}}.\label{3.13}
\end{flalign}
Therefore, combining \eqref{3.12} and \eqref{3.13}, we obtain that (also can see \cite{MMS})
\begin{equation}\label{3.14}
A\cdot\nabla R_{0}:L^{2,6}\rightarrow L^{2,6}.
\end{equation}
Define operator $T\in B(L^{2,6},L^{2,6})$ for $f\in L^{2,6}$ by
\begin{equation}\label{3.15}
Tf=(|A|^{2}-iA\cdot\nabla-i\nabla\cdot A)R_{0}f,
\end{equation}
then $(I+T)^{-1}$ exists in $B(L^{2,6},L^{2,6})$. Indeed, using the resolvent equation
\begin{equation}\label{3.16}
R_{A}+R_{A}(|A|^{2}-iA\cdot\nabla-i\nabla\cdot A)R_{0}=R_{0},
\end{equation}
and \eqref{3.15}, it follows that for any $f\in L^{2}$ and $g:=(I+T)f$, we have
$$R_{A}(I+T)f=R_{0}f,$$
i.e.
\begin{equation}\label{3.17}
R_{A}g=R_{0}f.
\end{equation}
Letting $f$ vary on $L^{2}$, it follows from \eqref{3.17} that the range of $(I+T)\Subset L^{2}$, which implies that the range of $(I+T)=L^{2,6}$. From this it follows by well know results on compact operators in a Hilbert space (the Fredholm-Riesz theory) that the inverse $(I+T)^{-1}$ exists in $B(L^{2,6},L^{2,6})$, that is to say (also can see the proof of Lemma 3.5 in \cite{KK})
\begin{equation}\label{3.18}
\|\langle x\rangle^{6}(I+A\cdot\nabla R_{0})^{-1}\langle x\rangle^{-6}\|_{L^{2}\rightarrow L^{2}}\leqslant C.
\end{equation}
Furthermore, applying estimate $\|R_{0}\langle x\rangle^{-4}\|_{L^{r}\rightarrow L^{r}}\lesssim\tau^{-\frac{1}{r}}$ for $1\leqslant r\leqslant 2$ from Corollary \ref{corollary3.1} (also can see Lemma 2.2 in \cite{LZ}), then we have
\begin{flalign}
\|R_{0}\langle x\rangle^{-6}\|_{L^{2}\rightarrow L^{r}}&\lesssim\|R_{0}\langle x\rangle^{-4}\|_{L^{r}\rightarrow L^{r}}\cdot\|\langle x\rangle^{-2}\|_{ L^{2}\rightarrow L^{r}}
\nonumber\\
&\lesssim\tau^{-\frac{1}{r}}.\label{3.19}
\end{flalign}
Consequently, according to \eqref{3.18} and \eqref{3.19}, we obtain that
\begin{center}
\uppercase\expandafter{\romannumeral1}:=$\|(\tau-\Delta_{A})^{-1}\langle x \rangle^{-6}\|_{L^{2}\rightarrow L^{r}}\lesssim \tau^{-\frac{1}{r}},$
\end{center}
where $1\leqslant r\leqslant2$.\\
$Setp\ 2.\ Estimate$ \uppercase\expandafter{\romannumeral2}.\\
Applying the estimates in $Setp\ 1$, we claim that
\begin{equation}\label{3.20}
\|A_{j}(\tau-\Delta_{A})^{-1}\partial_{j}f\|_{L^{2}}\leqslant C\tau^{-\frac{1}{2}}\|\partial_{j}f\|_{L^{2}},\ j=1,2,3.
\end{equation}
In fact,
\begin{equation*}
\begin{split}
&\|A_{j}(\tau-\Delta_{A})^{-1}f\|_{L^{2}}\\
\leqslant&\delta\|\langle x\rangle^{-12}R_{0}(I+A\cdot\nabla R_{0})^{-1}f\|_{L^{2}}\\
\leqslant&\delta\|\langle x\rangle^{-12}R_{0}\|_{L^{2}\rightarrow L^{2}}\cdot\|(I+A\cdot\nabla R_{0})^{-1}f\|_{L^{2}}\\
\leqslant&C\tau^{-\frac{1}{2}}\|(I+A\cdot\nabla R_{0})^{-1}f\|_{L^{2}}.
\end{split}
\end{equation*}
And let $(I+A\cdot\nabla R_{0})^{-1}f=g$ i.e. $f=g+A\cdot\nabla R_{0}g$, we have \begin{equation*}
\|g\|_{L^{2}}=\|f-A\cdot\nabla R_{0}g\|_{L^{2}}\leqslant \|f\|_{L^{2}}+\|A\cdot\nabla R_{0}g\|_{L^{2}}.
\end{equation*}
It is obvious that $\|A\cdot\nabla R_{0}g\|_{L^{2}}\leqslant\|A\|_{L^{3}_{x}}\|g\|_{L^{2}}\leqslant\delta\|\langle x\rangle^{-12}\|_{L^{3}_{x}}\cdot\|g\|_{L^{2}}$. Therefore we have $\|g\|_{L^{2}}\leqslant C\|f\|_{L^{2}}$, which implies \eqref{3.20}.\\
In order to estimate \uppercase\expandafter{\romannumeral2}, we need to prove the following inequality:
\begin{equation}\label{3.21}
\|A_{j}[\partial_{j},(\tau-\Delta_{A})^{-1}]f\|_{L^{2}}\leqslant C\tau^{-\frac{1}{2}}\|\partial_{j}f\|_{L^{2}}.
\end{equation}
In fact,
\begin{equation*}
\begin{split}
&[\partial_{j},(\tau-\Delta_{A})^{-1}]f\\
=&(\partial_{j}(\tau-\Delta_{A})^{-1}-(\tau-\Delta_{A})^{-1}\partial_{j})f\\
=&(\tau-\Delta_{A})^{-1}((\tau-\Delta_{A})\partial_{j}-\partial_{j}(\tau-\Delta_{A}))(\tau-\Delta_{A})^{-1}f\\
=&(\tau-\Delta_{A})^{-1}(2i(\partial_{j}A_{j}\cdot\partial_{j})+i(\partial_{jj}^{2}A_{j})-(\partial_{j}|A|^{2}))(\tau-\Delta_{A})^{-1}f.
\end{split}
\end{equation*}
Therefore,
\begin{flalign}
&\|A_{j}\partial_{j}(\tau-\Delta_{A})^{-1}f\|_{L^{2}}\nonumber\\
\leqslant&\|A_{j}(\tau-\Delta_{A})^{-1}\partial_{j}f\|_{L^{2}}+\|A_{j}[\partial_{j},(\tau-\Delta_{A})^{-1}]f\|_{L^{2}}\nonumber\\
\leqslant&C\tau^{-\frac{1}{2}}\|\partial_{j}f\|_{L^{2}}+\|A_{j}(\tau-\Delta_{A})^{-1}(2i(\partial_{j}A_{j})\cdot\partial_{j}+i(\partial_{jj}^{2}A_{j})\nonumber\\
&-(\partial_{j}|A|^{2}))(\tau-\Delta_{A})^{-1}f\|_{L^{2}}\nonumber\\
\leqslant&C\tau^{-\frac{1}{2}}\|\partial_{j}f\|_{L^{2}}+\|A_{j}(\tau-\Delta_{A})^{-1}\cdot2i(\partial_{j}A_{j})\cdot\partial_{j}(\tau-\Delta_{A})^{-1}f\|_{L^{2}}\nonumber\\
&+\|A_{j}(\tau-\Delta_{A})^{-1}(i(\partial_{jj}^{2}A_{j})-(\partial_{j}|A|^{2}))(\tau-\Delta_{A})^{-1}f\|_{L^{2}}\nonumber\\
:=&C\tau^{-\frac{1}{2}}\|\partial_{j}f\|_{L^{2}}+\textcircled{1}+\textcircled{2}.\label{3.22}
\end{flalign}
Next we will estimate $\textcircled{1}$ and $\textcircled{2}$ respectively. For $\textcircled{1}$ we obtain that
\begin{flalign}
&\|A_{j}(\tau-\Delta_{A})^{-1}\cdot2i(\partial_{j}A_{j})\cdot\partial_{j}(\tau-\Delta_{A})^{-1}f\|_{L^{2}}\nonumber\\
\leqslant&\delta^{2}\|\langle x\rangle^{-\frac{13}{2}}\cdot\langle x\rangle^{-\frac{11}{2}}R_{A}\langle x\rangle^{-\frac{11}{2}}\cdot\langle x\rangle^{-\frac{13}{2}}\cdot\partial_{j}(\tau-\Delta_{A})^{-1}f\|_{L^{2}}\nonumber\\
\leqslant&\delta^{2}\|\langle x\rangle^{-\frac{13}{2}}\|_{L^{\infty}}\cdot\|R_{A}\|_{L^{2,\frac{11}{2}}\rightarrow L ^{2,-\frac{11}{2}}}\cdot\|\langle x\rangle^{-\frac{13}{2}}\partial_{j}(\tau-\Delta_{A})^{-1}f\|_{L^{2}}\nonumber\\
\leqslant&\delta^{2}\|\langle x\rangle^{-\frac{13}{2}}\partial_{j}(\tau-\Delta_{A})^{-1}f\|_{L^{2}}\nonumber\\
\leqslant&C\tau^{-\frac{1}{2}}\|\partial_{j}f\|_{L^{2}}+\delta^{2}\textcircled{1}+\delta^{2}\textcircled{2}, \label{3.23}
\end{flalign}
and for $\textcircled{2}$
\begin{flalign}
&\|A_{j}(\tau-\Delta_{A})^{-1}(i(\partial_{jj}^{2}A_{j})-(\partial_{j}|A|^{2}))(\tau-\Delta_{A})^{-1}f\|_{L^{2}}\nonumber\\
\leqslant&\delta^{2}\|R_{A}\|_{L^{2,\frac{11}{2}}-L^{2,-\frac{11}{2}}}\cdot\|\langle x\rangle^{-\frac{13}{2}}(\tau-\Delta_{A})^{-1}f\|_{L^{2}}\nonumber\\
\leqslant&\delta^{2}\|\langle x\rangle^{-\frac{13}{4}}(\tau-\Delta_{A})^{-1}\langle x\rangle^{-\frac{13}{4}}f\|_{L^{2}}\nonumber\\
+&\delta^{2}\|\langle x\rangle^{-\frac{13}{4}}\|[\langle x\rangle^{-\frac{13}{4}},(\tau-\Delta_{A})^{-1}]f\|_{L^{2}}\nonumber\\
\leqslant&C\tau^{-\frac{1}{2}}\|\langle x\rangle^{-\frac{13}{4}}f\|_{L^{2}}+\delta^{2}\|\langle x\rangle^{-\frac{13}{4}}(\tau-\Delta_{A})^{-1}(A'\cdot\nabla+A'')(\tau-\Delta_{A})^{-1}f\|_{L^{2}}\nonumber\\
\leqslant&C\tau^{-\frac{1}{2}}\|\nabla f\|_{L^{2}}+\delta^{2}\textcircled{1}+\delta^{2}\textcircled{2},\label{3.24}
\end{flalign}
where $|A'|<|A|$ and $|A''|<|A|$.\\
Let us plus \eqref{3.23} and \eqref{3.24} as the following that
\begin{equation}\label{3.25}
\textcircled{1}+\textcircled{2}\leqslant C\tau^{-\frac{1}{2}}\|\nabla f\|_{L^{2}}+\delta^{2}\textcircled{1}+\delta^{2}\textcircled{2},
\end{equation}
hence, we have
\begin{equation}\label{3.26}
\textcircled{1}+\textcircled{2}\leqslant C\tau^{-\frac{1}{2}}\|\nabla f\|_{L^{2}}.
\end{equation}
According to \eqref{3.22} and \eqref{3.26}, we obtain that
\begin{equation}\label{3.27}
\|A_{j}\partial_{j}(\tau-\Delta_{A})^{-1}f\|_{L^{2}}\leqslant C\tau^{-\frac{1}{2}}\|\nabla f\|_{L^{2}}.
\end{equation}
Therefore,
\begin{center}
\uppercase\expandafter{\romannumeral2}:=$\|\langle x \rangle^{6}\tilde{\tilde A}\cdot\nabla(\tau-\Delta_{A})^{-1} f\|_{L^{2}}\lesssim \tau^{-\frac{1}{2}}\|\nabla f\|_{L^{2}}:=\tau^{-\frac{1}{2}}\|f\|_{\dot{H}^{1}_{x}}.$
\end{center}
Consequently, we obtain from $Setp\ 1$ and  $Setp\ 2$ that for $\tau\in(0,1)$
\begin{equation}\label{3.28}
\|(\tau-\Delta_{A})^{-1}\tilde{\tilde A}\cdot\nabla(\tau-\Delta_{A})^{-1}f\|_{L_{x}^{r}}\lesssim\tau^{-(\frac{1}{2}+\frac{1}{r})}\|f\|_{\dot{H}^{1}_{x}}, 1\leqslant r\leqslant2.
\end{equation}
On the other hand, we have
\begin{flalign}
&\|(\tau-\Delta_{A})^{-1}\tilde{A}(\tau-\Delta_{A})^{-1}f\|_{L_{x}^{r}}\nonumber\\
\leqslant&\delta\|(\tau-\Delta_{A})^{-1}\langle x\rangle^{-12}(\tau-\Delta_{A})^{-1}f\|_{L_{x}^{r}}\nonumber\\
\leqslant&\delta\|R_{0}\langle x\rangle^{-6}\|_{L^{2}\rightarrow L^{r}}\cdot\|\langle x\rangle^{6}(I+A\cdot\nabla R_{0})^{-1}\langle x\rangle^{-6}\|_{L^{2}\rightarrow L^{2}}\nonumber\\
&\cdot\|\langle x\rangle^{-6}(\tau-\Delta_{A})^{-1}f\|_{L_{x}^{2}}\nonumber\\
\leqslant&C\delta\tau^{-\frac{1}{r}}\|\langle x\rangle^{-6}(\tau-\Delta_{A})^{-1}f\|_{L_{x}^{2}}.\label{3.29}
\end{flalign}
Applying the proof in \eqref{3.24} to \eqref{3.29}, we obtain that
\begin{flalign}\label{3.30}
\|\langle x\rangle^{-6}(\tau-\Delta_{A})^{-1}f\|_{L_{x}^{2}}\leqslant C\tau^{-\frac{1}{2}}\|\nabla f\|_{L^{2}_{x}}.
\end{flalign}
Combining \eqref{3.29} and \eqref{3.30},
\begin{flalign}\label{3.31}
\tau^{-\frac{1}{r}}\|\langle x\rangle^{-6}(\tau-\Delta_{A})^{-1}f\|_{L_{x}^{2}}\lesssim\tau^{-(\frac{1}{r}+\frac{1}{2})}\|f\|_{\dot{H}^{1}}.
\end{flalign}
Therefore, from  \eqref{3.30} and \eqref{3.31}, it follows that for $\tau\in(0,1)$
\begin{equation}\label{3.32}
\|(\tau-\Delta_{A})^{-1}\tilde{A}(\tau-\Delta_{A})^{-1}f\|_{L_{x}^{r}}\lesssim\tau^{-(\frac{1}{r}+\frac{1}{2})}\|f\|_{\dot{H}^{1}}.
\end{equation}
When $\tau\in(1,\infty)$, On the one hand,
\begin{equation*}
\begin{split}
&\|(\tau-\Delta_{A})^{-1}\tilde{\tilde A}\cdot\nabla(\tau-\Delta_{A})^{-1}f\|_{L_{x}^{r}}\\
\lesssim &\|(\tau-\Delta_{A})^{-1}\|_{L^{r}\rightarrow L^{r}}\cdot\| \tilde{\tilde A}\|_{L^{2}\rightarrow L^{r}}\cdot\|\nabla(\tau-\Delta_{A})^{-1}\|_{\dot{H}^{1}\rightarrow L^{2}}\cdot\|\nabla f\|_{L^{2}}\\
\end{split}
\end{equation*}
$\ \ \ \ \ \ \  $:=\uppercase\expandafter{\romannumeral3}$\cdot$\uppercase\expandafter{\romannumeral4}$\cdot$\uppercase\expandafter{\romannumeral5}$\cdot\|\nabla f\|_{L^{2}}.$\\
$Setp\ 3.\ Estimate$ \uppercase\expandafter{\romannumeral3}.\\
In fact, for $\tau>0$ and $1\leqslant r< \infty$, from Lemma \ref{lemma3.1} (see also Lemma 2.2 in \cite{LZ}), we have the following estimates
\begin{flalign}
\|(\tau-\Delta)^{-1}f\|_{L^{r}}\lesssim \tau^{-1}\|f\|_{L^{r}}.\label{3.33}
\end{flalign}
Using the interpolation theorem, we have
\begin{flalign}
\|(\tau-\Delta)f\|_{L^{r}}\lesssim \|(\tau-\Delta_{A})f\|_{L^{r}},\,\, {\rm for}\  \tau>1.\label{3.34}
\end{flalign}
Hence, from \eqref{3.33} and \eqref{3.34}, it is obtained that
\begin{center}
\uppercase\expandafter{\romannumeral3}:=$\|(\tau-\Delta_{A})^{-1}\|_{L^{r}\rightarrow L^{r}}\lesssim \tau^{-1}.$
\end{center}
$Setp\ 4.\ Estimate$ \uppercase\expandafter{\romannumeral5}.\\
For $1\leqslant r\leqslant2$,
\begin{equation*}
\begin{split}
\|\tilde{\tilde A}f\|_{L^{r}}&\lesssim (\int_{R^{3}}\langle x\rangle^{-12r}|f|^{r}dx)^{\frac{1}{r}}\\
&\lesssim \|f\|_{L^{2}}\cdot(\int_{R^{3}}\langle x\rangle^{-\frac{24r}{2-r}}dx)^{\frac{1}{r}\cdot\frac{2-r}{2}}\\
&\lesssim \|f\|_{L^{2}}.
\end{split}
\end{equation*}
Therefore,
\begin{center}
\uppercase\expandafter{\romannumeral4}:=$\|\tilde{\tilde A}\|_{L^{2}\rightarrow L^{r}}\leqslant C.
$\end{center}
$Setp\ 5.\ Estimate$ \uppercase\expandafter{\romannumeral5}.\\
From estimate \eqref{3.33}, it is obtained that
\begin{equation}\label{3.35}
\|\nabla(\tau-\Delta)^{-1}f\|_{L^{2}}=\|(\tau-\Delta)^{-1}\nabla f\|_{L^{2}}\lesssim \tau^{-1}\|\nabla f\|_{L^{2}}.
\end{equation}
Furthermore,
\begin{equation*}
\begin{split}
&\|(\tau-\Delta_{A})(\nabla f)\|_{L^{2}}\\
=&\|(\tau-\Delta)(\nabla f)-iA\cdot \nabla (\nabla f)+(|A|^{2}-i\nabla\cdot A)(\nabla f)\|_{L^{2}}\\
\geqslant&\|(\tau-\Delta)(\nabla f)\|_{L^{2}}-C\|A\cdot \nabla (\nabla f)\|_{L^{2}},
\end{split}
\end{equation*}
and
$$\|A\cdot \nabla (\nabla f)\|_{L^{2}}\leqslant C\|(-\Delta)(\nabla f)\|_{L^{2}}+\dfrac{1}{C}\|\nabla f\|_{L^{2}}.$$
Hence, we have
\begin{equation}\label{3.36}
\|(\tau-\Delta)(\nabla f)\|_{L^{2}}\leqslant C\|(\tau-\Delta_{A})(\nabla f)\|_{L^{2}}.
\end{equation}
According to \eqref{3.35} and \eqref{3.36}, we obtain that
\begin{equation*}
\|\nabla f\|_{L^{2}}\lesssim\tau^{-1}\|(\tau-\Delta)(\nabla f)\|_{L^{2}}\lesssim\tau^{-1}\|(\tau-\Delta_{A})(\nabla f)\|_{L^{2}},
\end{equation*}
and applying a similar way of $Setp\ 2$, we have
\begin{equation*}\begin{split}
&\|\nabla(\tau-\Delta_{A})^{-1}f\|_{L^{2}}\\
\leqslant&C\tau^{-1}\| (\tau-\Delta_{A})\nabla(\tau-\Delta_{A})^{-1}f\|_{L^{2}}\\
\leqslant&C\tau^{-1}\|\nabla f\|_{L^{2}}+\tau^{-1}\|(\tau-\Delta_{A})[\nabla,(\tau-\Delta_{A})^{-1}]f\|_{L^{2}}\\
\leqslant&C\tau^{-1}\|\nabla f\|_{L^{2}}+\tau^{-1}\|(A'\cdot\nabla+A'') (\tau-\Delta_{A})^{-1}f\|_{L^{2}}\\
\leqslant&C\tau^{-1}\|\nabla f\|_{L^{2}}.
\end{split}
\end{equation*}
Hence,
\begin{center}\uppercase\expandafter{\romannumeral5}:=$\|\nabla(\tau-\Delta_{A})^{-1}\|_{\dot{H}^{1}\rightarrow L^{2}}\lesssim\tau^{-1}.$
\end{center}
Consequently, we obtain from $Setp\ 3$-$Setp\ 5$ that for $\tau\in(1,\infty),$
\begin{equation}\label{3.37}
\|(\tau-\Delta_{A})^{-1}\tilde{\tilde A}\cdot\nabla(\tau-\Delta_{A})^{-1}f\|_{L_{x}^{r}}\lesssim\tau^{-2}\|f\|_{\dot{H}^{1}_{x}}, 1\leqslant r\leqslant 2.
\end{equation}
On the other hand, we have
\begin{flalign}
&\|(\tau-\Delta_{A})^{-1}\tilde{A}(\tau-\Delta_{A})^{-1}f\|_{L_{x}^{r}}\nonumber\\
\lesssim&\|(\tau-\Delta_{A})^{-1}\langle x\rangle^{-12}(\tau-\Delta_{A})^{-1}f\|_{L_{x}^{r}}\nonumber\\
\lesssim&\tau^{-1}\|\langle x\rangle^{-12}(\tau-\Delta_{A})^{-1}f\|_{L_{x}^{r}}\nonumber\\
\lesssim&\tau^{-1}\|\langle x\rangle^{-12}\|_{L^{\frac{6r}{6-r}}}\cdot\|(\tau-\Delta_{A})^{-1}f\|_{L_{x}^{r}}\nonumber\\
\lesssim&\tau^{-1}\|\nabla(\tau-\Delta_{A})^{-1}f\|_{L_{x}^{r}}\nonumber\\
\lesssim&\tau^{-2}\|f\|_{\dot{H}^{1}_{x}}.\label{3.38}
\end{flalign}
Therefore, according to \eqref{3.28}, \eqref{3.32}, \eqref{3.37} and \eqref{3.38}, we obtain that
\begin{equation*}
\begin{split}
\|V(s)f\|_{L^{r}_{x}}\leqslant& C\int^{1}_{0}\tau^{\frac{s}{2}}\|(\tau-\Delta_{A})^{-1}(\tilde{A}-\tilde{\tilde A}\cdot\nabla)(\tau-\Delta_{A})^{-1}f\|_{L_{x}^{r}}d\tau\\
+&C'\int^{\infty}_{1}\tau^{\frac{s}{2}}\|(\tau-\Delta_{A})^{-1}(\tilde{A}-\tilde{\tilde A}\cdot\nabla)(\tau-\Delta_{A})^{-1}f\|_{L_{x}^{r}}d\tau\\
\leqslant& C\int^{1}_{0}\tau^{\frac{s}{2}-\frac{1}{2}-\frac{1}{r}}\|f\|_{\dot{H}^{1}_{x}}d\tau+C'\int^{\infty}_{1}\tau^{\frac{s}{2}-2}\|f\|_{\dot{H}^{1}_{x}}d\tau\\
\leqslant& C\|f\|_{\dot{H}^{1}_{x}}+C'\|f\|_{\dot{H}^{1}_{x}}\\
\lesssim&\|f\|_{\dot{H}^{1}_{x}},
\end{split}
\end{equation*}
where $1\leqslant r\leqslant 2$,  $\frac{2}{r}-1<s<2$. This completes the proof of Lemma \ref{lemma3.2}.
\end{proof}

\section{\bf Proof of Theorem \ref{theorem1.1}}

\begin{lemma}\label{lemma4.1}
(\textquotedblleft Almost equivalence\textquotedblright)\\
Let $A(x)$ satisfy hypothesis
$(\textbf H)$, then for $1\leqslant r\leqslant2$ and $0<s<2$, we have
\begin{equation}\label{4.1}
\|(-\Delta_{A})^{\frac{s}{2}}u-(-\Delta)^{\frac{s}{2}}u\|_{L^{r}_{x}(\mathbb{R}^{3})}\leqslant C\|u\|_{\dot{H}^{1}_{x}(\mathbb{R}^{3})}.
\end{equation}		
\end{lemma}

\begin{proof}
Applying representation \eqref{2.18} in Lemma \ref{lemma2.4}, we deduce that
\begin{equation*}
\begin{split}
&(-\Delta_{A})^{\frac{s}{2}}u\\
=&c(s)(-\Delta_{A})\int^{\infty}_{0}\tau^{\frac{s}{2}-1}(\tau-\Delta_{A})^{-1}ud\tau\\
=&c(s)(-\Delta_{A})\int^{\infty}_{0}\tau^{\frac{s}{2}-1}((\tau-\Delta_{A})^{-1}-(\tau-\Delta)^{-1})ud\tau\\
&+c(s)(-\Delta_{A})\int^{\infty}_{0}\tau^{\frac{s}{2}-1}(\tau-\Delta)^{-1}ud\tau\\
=&c(s)\Delta_{A}\int^{\infty}_{0}\tau^{\frac{s}{2}-1}(\tau-\Delta_{A})^{-1}(|A|^{2}-iA\cdot\nabla-i\nabla\cdot A)(\tau-\Delta)^{-1}ud\tau\\
&+c(s)(-\Delta_{A})\int^{\infty}_{0}\tau^{\frac{s}{2}-1}(\tau-\Delta)^{-1}ud\tau\\
=&-c(s)\int^{\infty}_{0}\tau^{\frac{s}{2}-1}(\tau-\Delta_{A})(\tau-\Delta_{A})^{-1}(|A|^{2}-iA\cdot\nabla-i\nabla\cdot A)(\tau-\Delta)^{-1}ud\tau\\
\end{split}
\end{equation*}	
\begin{equation*}
\begin{split}
&+c(s)(-\Delta_{A})\int^{\infty}_{0}\tau^{\frac{s}{2}-1}(\tau-\Delta)^{-1}ud\tau\\
&+c(s)\int^{\infty}_{0}\tau^{\frac{s}{2}}(\tau-\Delta_{A})^{-1}(|A|^{2}-iA\cdot\nabla-i\nabla\cdot A)(\tau-\Delta)^{-1}ud\tau\\=&-c(s)\int^{\infty}_{0}\tau^{\frac{s}{2}-1}(|A|^{2}-iA\cdot\nabla-i\nabla\cdot A)(\tau-\Delta)^{-1}ud\tau\\
&+c(s)\int^{\infty}_{0}\tau^{\frac{s}{2}}(\tau-\Delta_{A})^{-1}(|A|^{2}-iA\cdot\nabla-i\nabla\cdot A)(\tau-\Delta)^{-1}ud\tau\\
&+c(s)(-\Delta)\int^{\infty}_{0}\tau^{\frac{s}{2}-1}(\tau-\Delta)^{-1}ud\tau\\
&+c(s)\int^{\infty}_{0}\tau^{\frac{s}{2}-1}(|A|^{2}-iA\cdot\nabla-i\nabla\cdot A)(\tau-\Delta)^{-1}ud\tau\\
=&(-\Delta)^{\frac{s}{2}}u+c(s)\int^{\infty}_{0}\tau^{\frac{s}{2}}(\tau-\Delta_{A})^{-1}(|A|^{2}-iA\cdot\nabla-i\nabla\cdot A)(\tau-\Delta)^{-1}ud\tau.
\end{split}
\end{equation*}	
That is,
\begin{equation}\label{4.2}
(-\Delta_{A})^{\frac{s}{2}}u-(-\Delta)^{\frac{s}{2}}u=c(s)\int^{\infty}_{0}\tau^{\frac{s}{2}}(\tau-\Delta_{A})^{-1}V_{3}(x)(\tau-\Delta)^{-1}ud\tau,
\end{equation}
where $V_{3}(x):=|A|^{2}-iA\cdot\nabla-i\nabla\cdot A$ and  $(c(s))^{-1}=\int^{\infty}_{0}\tau^{\frac{s}{2}-1}(\tau+1)^{-1}d\tau$ for every $0<s<2$.\\
Applying the similar way from the proof of Lemma \ref{lemma3.2}, we obtain that
\begin{equation*}
\begin{split}
&\|(-\Delta_{A})^{\frac{s}{2}}u-(-\Delta)^{\frac{s}{2}}u\|_{L^{r}_{x}(\mathbb{R}^{3})}\\
\leqslant& c(s)\int^{\infty}_{0}\tau^{\frac{s}{2}}\|(\tau-\Delta_{A})^{-1}V_{3}(x)(\tau-\Delta)^{-1}u\|_{L^{r}_{x}(\mathbb{R}^{3})}d\tau\\
\leqslant& C\|u\|_{\dot{H}^{1}_{x}(\mathbb{R}^{3})},
\end{split}
\end{equation*}	
where $0<s<2$ and $1\leqslant r\leqslant 2$. This implies estimate \eqref{4.1}.
\end{proof}

\begin{lemma}\label{lemma4.2}
Let $|J|^{s}$ and $|J_{A}|^{s}$ as definition of \eqref{2.10} and \eqref{2.11} respectively. Then for $\forall\ \varepsilon\in(0,1)$ and $s\in(0,2)$, we have
\begin{equation}\label{4.3}
\||J_{A}|^{s}u\|_{L^{2}_{x}(\mathbb{R}^{3})}\leqslant Ct^{s+\varepsilon-\frac{3}{2}}(\||J|^{\frac{3}{2}-\varepsilon}u\|_{L^{2}_{x}(\mathbb{R}^{3})}+\||J|^{s}u\|_{L^{2}_{x}(\mathbb{R}^{3})}),
\end{equation}
\begin{equation}\label{4.4}
\||J|^{s}u\|_{L^{2}_{x}(\mathbb{R}^{3})}\leqslant Ct^{s+\varepsilon-\frac{3}{2}}(\||J_{A}|^{\frac{3}{2}-\varepsilon}u\|_{L^{2}_{x}(\mathbb{R}^{3})}+\||J_{A}|^{s}u\|_{L^{2}_{x}(\mathbb{R}^{3})}).
\end{equation}
\end{lemma}

\begin{proof}
According to Lemma \ref{lemma4.1}, we can exchange $-\Delta_{A}$ and $-\Delta$, $|J|^{s}$ and $|J_{A}|^{s}$, therefore we only just prove estimate \eqref{4.3}.\\
In fact,
\begin{equation}\label{4.5}
\|\sqrt{-\Delta_{A}}\,u\|^{2}_{L^{2}_{x}}\leqslant C_{1}\|\sqrt{-\Delta}\,u\|^{2}_{L^{2}_{x}}+C_{2}\|(\nabla\cdot A+A\cdot\nabla+|A|^{2})|u|^{2}\|_{L_{x}^{1}}.
\end{equation}
And for the second term on the right-hand side of \eqref{4.5}, by H\"older's inequality and Sobolev embedding, we obtain that
\begin{flalign}
&\|(\nabla\cdot A+A\cdot\nabla+|A|^{2})|u|^{2}\|_{L_{x}^{1}}\nonumber\\\leqslant& C\|(\nabla\cdot A+|A|^{2})|u|^{2}\|_{L^{1}_{x}}+C\|A\cdot\nabla |u|^{2}\|_{L^{1}_{x}}\nonumber\\
\leqslant& C\|\nabla\cdot A+|A|^{2}\|_{L_{x}^{m'_{1}}}\cdot\|u\|_{L_{x}^{2m_{1}}}^{2}+C\|\langle x\rangle^{-6}\|_{L_{x}^{m'_{2}}}\cdot\|\langle x\rangle^{6}A\cdot\nabla |u|^{2}\|_{L^{m_{2}}}\nonumber\\
\leqslant&C\|u\|_{L_{x}^{2m_{1}}}^{2}+C\|(-\Delta)^{\frac{1}{2}}u\|_{L_{x}^{2m_{2}}}^{2}\nonumber\\
\leqslant& C\|(-\Delta)^{\frac{k_{1}}{2}}u\|^{2}_{L^{2}}+C\|(-\Delta)^{\frac{k_{2}}{2}+\frac{1}{2}}u\|^{2}_{L^{2}}\nonumber\\
\leqslant& C\|(-\Delta)^{\frac{3}{2}(\frac{1}{2}-\eta)}u\|^{2}_{L_{x}^{2}},\label{4.6}
\end{flalign}	
where $\frac{1}{m_{i}}+\frac{1}{m'_{i}}=1$, $i=1,2$; $m_1=\frac{1}{2\eta},m_2=\frac{3}{2+6\eta}<1$; $k_{1}=\frac{3}{2}-3\eta$, $k_{2}=\frac{1}{2}-3\eta$ (similarly see Theorem 4.12 in \cite{AF}).\\
Then, combining \eqref{4.5} and \eqref{4.6}, we have
\begin{flalign}
\|\sqrt{-\Delta_{A}}\,u\|^{2}_{L^{2}_{x}}&\leqslant C_{1}\|\sqrt{-\Delta}\,u\|^{2}_{L^{2}_{x}}+C_{2}'\|(-\Delta)^{\frac{3}{2}(\frac{1}{2}-\eta)}u\|^{2}_{L_{x}^{2}}\nonumber\\
&\leqslant C\|(-\Delta)^{\frac{3}{4}-\frac{3}{2}\eta}(1+(-\Delta)^{-\frac{1}{4}+\frac{3}{2}\eta})u\|^{2}_{L_{x}^{2}}.\label{4.7}
\end{flalign}
Furthermore, for every $0<s_{1}<2$, from estimate \eqref{4.1} we obtain that
\begin{flalign}
&\|(-\Delta_{A})^{\frac{s_{1}}{2}}u\|^{2}_{L_{x}^{2}}\nonumber\\
\leqslant&\|(-\Delta_{A})^{\frac{s_{1}}{2}}u-(-\Delta)^{\frac{s_{1}}{2}}u\|_{L_{x}^{2}}+\|(-\Delta)^{\frac{s_{1}}{2}}u\|^{2}_{L_{x}^{2}}\nonumber\\
\leqslant&C\|u\|^{2}_{\dot{H}_{x}^{1}}+\|(-\Delta)^{\frac{s_{1}}{2}}u\|^{2}_{L_{x}^{2}}\nonumber\\
\leqslant&C\|(-\Delta)^{\frac{3}{4}-\frac{3}{2}\eta}(1+(-\Delta)^{-\frac{1}{4}+\frac{3}{2}\eta}+(-\Delta)^{-\frac{3}{4}+\frac{3}{2}\eta+\frac{s_{1}}{2}})u\|^{2}_{L_{x}^{2}}.\label{4.8}
\end{flalign}
Applying the interpolation method to inequalities \eqref{4.7} and \eqref{4.8},  for every $\frac{1}{2}<s<s_{1}<2$, we have
\begin{equation}\label{4.9}
\|(-\Delta_{A})^{\frac{s}{2}}u\|_{L_{x}^{2}}\leqslant C\|(-\Delta)^{\frac{3}{4}-\frac{3}{2}\eta}(1+(-\Delta)^{\frac{s}{2}-\frac{3}{4}+\frac{3}{2}\eta})u\|_{L_{x}^{2}}.
\end{equation}	
Multiplying both sides of inequality \eqref{4.9} by $t^{s}$ gives the estimate
$$
\||J_{A}|^{s}u\|_{L^{2}_{x}}\leqslant Ct^{s+3\eta-\frac{3}{2}}(\||J|^{\frac{3}{2}-3\eta}u\|_{L^{2}_{x}}+\||J|^{s}u\|_{L^{2}_{x}}),
$$
and letting $\varepsilon=3\eta$, \eqref{4.3} follows. The proof of Lemma \ref{lemma4.2} is completed.
\end{proof}
\begin{lemma}\label{lemma4.3}
Let $A(x)$ satisfy hypothesis $(\textbf H)$, $M(-t)$ and $|J_{A}|^{s}$ as in \eqref{2.2} and \eqref{2.11} respectively. Then for $\forall\ \varepsilon\in(0,1)$, $t\geqslant1$ and $1<s<2$, we have the following estimate:
\begin{equation}\label{4.10}
\|t^{s-1}V(s)M(-t)u\|_{L^{4}_{t}((1,\infty);L^{1}_{x}(\mathbb{R}^{3}))}\leqslant C_{s}(\||J_{A}|^{\frac{3}{2}-\varepsilon}u\|_{L^{\infty}_{t}L^{2}_{x}}+\||J_{A}|u\|_{L^{\infty}_{t}L^{2}_{x}}).
\end{equation}
\end{lemma}

\begin{proof}	
According to estimate \eqref{3.10} in Lemma \ref{lemma3.2}, it follows that
\begin{flalign}
&\|t^{s-1}V(s)M(-t)u\|_{L^{4}_{t}((1,\infty);L^{1}_{x}(\mathbb{R}^{3}))}\nonumber\\
\leqslant&C_{s}\|t^{s-1}\|V(s)M(-t)u\|_{L^{1}_{x}}\|_{L^{4}_{t}}\nonumber\\
\leqslant&C_{s}\|t^{s-1}\|M(-t)u\|_{\dot{H}^{1}_{x}}\|_{L^{4}_{t}}.\label{4.11}
\end{flalign}	
Furthermore,
\begin{flalign}
&\|M(-t)u\|_{\dot{H}^{1}_{x}(\mathbb{R}^{3})}\nonumber\\
=&\|(-\Delta)^{\frac{1}{2}}M(-t)u\|_{L^{2}_{x}}\nonumber\\
=&\|M(t)(-t^{2}\Delta)^{\frac{1}{2}}M(-t)u\|_{L^{2}_{x}}\cdot t^{-1}\nonumber\\
=&\||J|u\|_{L^{2}_{x}}\cdot t^{-1}.\label{4.12}
\end{flalign}	
Applying estimate \eqref{4.4} of Lemma \ref{lemma4.2} to equality \eqref{4.12}, we obtain that
\begin{equation}\label{4.13}
\|M(-t)u\|_{\dot{H}^{1}_{x}(\mathbb{R}^{3})}\leqslant Ct^{\varepsilon-\frac{3}{2}}(\||J_{A}|^{\frac{3}{2}-\varepsilon}u\|_{L^{2}_{x}}+\||J_{A}|u\|_{L^{2}_{x}}).
\end{equation}
Therefore, substituting \eqref{4.13} into \eqref{4.11} and letting $s<\frac{9}{4}-\varepsilon$ such that $t^{s-1+\varepsilon-\frac{3}{2}}\in L^{4}_{t}(1,\infty)$, estimate \eqref{4.10} is obtained from \eqref{4.13} directly. The proof of Lemma \ref{lemma4.3} is completed.
\end{proof}

\begin{lemma}\label{lemma4.4}
Let $A(x)$ satisfy hypothesis $(\textbf H)$, $|J|^{s}$ and $|J_{A}|^{s}$ as in \eqref{2.10} and \eqref{2.11} respectively. Then for $\forall\varepsilon\in(0,1)$, $s\in(\frac{3}{2},\frac{5}{3})$ and  $p>\frac{5}{3}$, we have the following estimate:
\begin{equation}\label{4.14}
\begin{split}
\||J_{A}|^{s}(|u|^{p-1}u)\|_{L_{t}^{1}L_{x}^{2}}\lesssim&\|u_{0}\|^{(1-\frac{3}{2s})(p-1)}_{L_{x}^{2}}\cdot(\||J_{A}|^{s}u\|_{L_{t}^{\infty}L_{x}^{2}}+\||J_{A}|^{\frac{3}{2}-\varepsilon}u\|_{L_{t}^{\infty}L_{x}^{2}})^{\frac{3(p-1)}{2s}}\\\\
\cdot&(\||J_{A}|^{\frac{3}{2}-\varepsilon}u\|_{L_{t}^{\infty}L_{x}^{2}}+\||J_{A}|^{s}u\|_{L_{t}^{\infty}L_{x}^{2}}+\||J_{A}|u\|_{L_{t}^{\infty}L_{x}^{2}}).
\end{split}
\end{equation}
\end{lemma}

\begin{proof}
Let $v=e^{\frac{-i|x|^{2}}{4t}}u$, we have the following estimate for $0<s<\frac{5}{3}$, $p>\frac{5}{3}$ (see also Lemma 2.3 in \cite{HN} and Lemma 3.4 in \cite{GOV})
\begin{flalign}
&\||J|^{s}(|u|^{p-1}u)\|_{L^{2}_{x}}\nonumber\\
=&t^{s}\cdot\||v|^{p-1}v\|_{\dot{H}^{s}_{x}}\nonumber\\
\leqslant&C\|v\|^{p-1}_{L_{x}^{\infty}}\cdot t^{s}\cdot\|v\|_{\dot{H}^{s}_{x}}\nonumber\\
\leqslant&C\|u\|^{p-1}_{L_{x}^{\infty}}\cdot\|(-t^{2}\Delta)^{\frac{s}{2}}e^{\frac{-i|x|^{2}}{4t}}u\|_{L_{x}^{2}}\nonumber\\
\leqslant&C\|u\|^{p-1}_{L_{x}^{\infty}}\cdot\||J|^{s}u\|_{L^{2}_{x}}.\label{4.15}
\end{flalign}
Hence, applying Lemma \ref{lemma4.1}, Lemma \ref{lemma4.2} and inequality \eqref{4.12}, we have
\begin{equation*}
\begin{split}
&\||J_{A}|^{s}(|u|^{p-1}u)\|_{L_{t'}^{1}((1,t);L_{x}^{2}(\mathbb{R}^{3}))}\\
\leqslant&C\|\langle t'\rangle^{s+\varepsilon-\frac{3}{2}}(\||J|^{\frac{3}{2}-\varepsilon}(|u|^{p-1}u)\|_{L_{x}^{2}}+\||J|^{s}(|u|^{p-1}u)\|_{L_{x}^{2}})\|_{L_{t'}^{1}}\\
\leqslant&C\|\langle t'\rangle^{s+\varepsilon-\frac{3}{2}}\cdot\|u\|^{p-1}_{L_{x}^{\infty}}\cdot(\||J|^{\frac{3}{2}-\varepsilon}u\|_{L_{x}^{2}}+\||J|^{s}u\|_{L_{x}^{2}})\|_{L_{t'}^{1}}\\
\leqslant&C\|\langle t'\rangle^{s+\varepsilon-\frac{3}{2}}\cdot\|u\|^{p-1}_{L_{x}^{\infty}}\cdot(\| |J|^{\frac{3}{2}-\varepsilon}u-|J_{A}|^{\frac{3}{2}-\varepsilon}u\|_{L_{x}^{2}}+\||J_{A}|^{\frac{3}{2}-\varepsilon}u\|_{L_{x}^{2}}\\
&+\||J|^{s}u-|J_{A}|^{s}u\|_{L_{x}^{2}}+\||J_{A}|^{s}u\|_{L_{x}^{2}})\|_{L_{t'}^{1}}\\
\leqslant&C\|\langle t'\rangle^{s+\varepsilon-\frac{3}{2}}\cdot\|u\|^{p-1}_{L_{x}^{\infty}}\cdot(\langle t'\rangle^{\frac{3}{2}-\varepsilon}\|M(-t')u\|_{\dot{H}_{x}^{1}}+\||J_{A}|^{\frac{3}{2}-\varepsilon}u\|_{L_{x}^{2}}\\
&+\langle t'\rangle^{s}\|M(-t')u\|_{\dot{H}_{x}^{1}}+\||J_{A}|^{s}u\|_{L_{x}^{2}})\|_{L_{t'}^{1}}\\
\end{split}
\end{equation*}	
\begin{equation*}
\begin{split}
\leqslant&C\|\langle t'\rangle^{s+\varepsilon-\frac{3}{2}}\cdot\|u\|^{p-1}_{L_{x}^{\infty}}\cdot[\langle t'\rangle^{\frac{3}{2}-\varepsilon+\varepsilon-\frac{3}{2}}(\||J_{A}|^{\frac{3}{2}-\varepsilon}u\|_{L_{x}^{2}}+\||J_{A}|u\|_{L_{x}^{2}})+\||J_{A}|^{\frac{3}{2}-\varepsilon}u\|_{L_{x}^{2}}\\
&+\langle t'\rangle^{s+\varepsilon-\frac{3}{2}}(\||J_{A}|^{\frac{3}{2}-\varepsilon}u\|_{L_{x}^{2}}+\||J_{A}|u\|_{L_{x}^{2}})+\||J_{A}|^{s}u\|_{L_{x}^{2}}]\|_{L_{t'}^{1}}\\
\leqslant&C\|\langle t'\rangle^{s+\varepsilon-\frac{3}{2}}\cdot\|u\|^{p-1}_{L_{x}^{\infty}}\cdot(2\||J_{A}|^{\frac{3}{2}-\varepsilon}u\|_{L_{x}^{2}}+\||J_{A}|u\|_{L_{x}^{2}}+\||J_{A}|^{s}u\|_{L_{x}^{2}}\\
&+\langle t'\rangle^{s+\varepsilon-\frac{3}{2}}(\||J_{A}|^{\frac{3}{2}-\varepsilon}u\|_{L_{x}^{2}}+\||J_{A}|u\|_{L_{x}^{2}}))\|_{L_{t'}^{1}}.
\end{split}
\end{equation*}	
And then, again applying Lemma \ref{lemma4.2}, we can continue the estimate as following:
\begin{flalign}
&......\nonumber\\
\leqslant&C\|\langle t'\rangle^{2s+2\varepsilon-3}\cdot\|u\|^{p-1}_{L_{x}^{\infty}}\cdot
(\||J_{A}|^{\frac{3}{2}-\varepsilon}u\|_{L_{x}^{2}}+\||J_{A}|^{s}u\|_{L_{x}^{2}}+\||J_{A}|u\|_{L_{x}^{2}})\|_{L_{t'}^{1}}\nonumber\\
\leqslant&C\int_{1}^{t}\langle t'\rangle^{2s+2\varepsilon-3+\frac{3(p-1)}{2s}(\varepsilon-\frac{3}{2})}\cdot\|u\|_{L_{x}^{2}}^{(1-\frac{3}{2s})(p-1)}\cdot(\||J_{A}|^{s}u\|_{L_{x}^{2}}+\||J_{A}|^{\frac{3}{2}-\varepsilon}u\|_{L_{x}^{2}})^{\frac{3(p-1)}{2s}}\nonumber\\
&\cdot(\||J_{A}|^{\frac{3}{2}-\varepsilon}u\|_{L_{x}^{2}}+\||J_{A}|^{s}u\|_{L_{x}^{2}}+\||J_{A}|u\|_{L_{x}^{2}})dt'.\label{4.16}
\end{flalign}
Since $p>\frac{5}{3}$, we can choose $s>\frac{3}{2}$ and $\varepsilon>0$ such that
\begin{equation}\label{4.17}
\int_{1}^{t}\langle t'\rangle^{2s+2\varepsilon-3+\frac{3(p-1)}{2s}(\varepsilon-\frac{3}{2})}dt'<C.
\end{equation}
Applying \eqref{4.17} to \eqref{4.16}, estimate \eqref{4.14} follows. The proof of Lemma \ref{lemma4.4} is completed.
\end{proof}

\hspace{-5mm}\textbf{Proof of Theorem \ref{theorem1.1}:}
According to Lemma \ref{lemma2.1}, we have the following equation:
\begin{equation}\label{4.18}
(i\partial_{t}+\Delta_{A})(|J_{A}|^{s}u)-it^{s-1}M(t)V(s)M(-t)u+\rho|J_{A}|^{s}(|u|^{p-1}u)=0.
\end{equation}
From Strichartz estimates for solution of equation \eqref{4.18}, we obtain that
\begin{equation*}
\begin{split}
&\||J_{A}|^{s}u\|_{L_{t}^{\infty}((1,T);L_{x}^{2}(\mathbb{R}^{3}))}\\
\leqslant&C_{1}\||J_{A}|^{s}u(1)\|_{L_{x}^{2}(\mathbb{R}^{3})}\\
+&C_{s}\|t^{s-1}M(t)V(s)M(-t)u\|_{L_{t}^{4}((1,T);L_{x}^{1}(\mathbb{R}^{3}))}\\
+&C_{2}\||J_{A}|^{s}(|u|^{p-1}u)\|_{L_{t}^{1}((1,T);L_{x}^{2}(\mathbb{R}^{3}))}.
\end{split}
\end{equation*}	
Applying Lemma \ref{lemma4.3} and Lemma \ref{lemma4.4}, it follows that for $\forall s_{0}\in(\frac{3}{2},\frac{5}{3})$ and $p\in(\frac{5}{3},5)$,
\begin{flalign}
&\||J_{A}|^{s_{0}}u\|_{L_{t}^{\infty}((1,T);L_{x}^{2}(\mathbb{R}^{3}))}\nonumber\\
\leqslant&C_{1}\||J_{A}|^{s_{0}}u(1)\|_{L_{x}^{2}}+C_{s}(\||J_{A}|^{\frac{3}{2}-\varepsilon}u\|_{L^{\infty}_{t}L^{2}_{x}}+\||J_{A}|u\|_{L^{\infty}_{t}L^{2}_{x}})\nonumber\\
+&C_{2}\|u_{0}\|^{(p-1)\cdot (1-\frac{3}{2s_{0}})}_{L_{x}^{2}}\cdot(\||J_{A}|^{s_{0}}u\|_{L_{t}^{\infty}L_{x}^{2}}+\||J_{A}|^{\frac{3}{2}-\varepsilon}u\|_{L_{t}^{\infty}L_{x}^{2}})^{\frac{3(p-1)}{2s_{0}}}
\nonumber\\
\cdot&(\||J_{A}|^{\frac{3}{2}-\varepsilon}u\|_{L_{t}^{\infty}L_{x}^{2}}+\||J_{A}|^{s_{0}}u\|_{L_{t}^{\infty}L_{x}^{2}}+\||J_{A}|u\|_{L_{t}^{\infty}L_{x}^{2}}).\label{4.19}
\end{flalign}
And let $\|u_{0}\|_{\Sigma_{s_{0}}}^{2}:=\|u_{0}\|_{H_{x}^{s_{0}}}^{2}+\||x|^{s_{0}}u_{0}\|^{2}_{L^{2}_{x}}+\||x|^{s_{2}}(-\Delta)^{\frac{s_{1}}{2}}u\|^{2}_{L^{2}_{x}(\mathbb{R}^{3})}$ ($0<s_{1},s_{2}<s_{0}=s_{1}+s_{2}$), we claim that
\begin{equation}\label{4.20}
\||J_{A}|^{s_{0}}u\|_{L_{t}^{\infty}((1,\infty);L_{x}^{2}(\mathbb{R}^{3}))}\leqslant C\|u_{0}\|_{\Sigma_{s_{0}}}.
\end{equation}
In fact,
\begin{flalign}
&\||J_{A}|^{s_{0}}u(1)\|_{L_{x}^{2}}\nonumber\\
\leqslant&C(\||J|^{\frac{3}{2}-\varepsilon}u(1)\|_{L^{2}_{x}}+\||J|^{s_{0}}u(1)\|_{L^{2}_{x}})\nonumber\\
\leqslant&C(\|u_{0}\|_{L^{2}_{x}}+\||J|^{s_{0}}u(1)\|_{L^{2}_{x}})\nonumber\\
\leqslant&C\|u_{0}\|_{\Sigma_{s_{0}}}.\label{4.21}
\end{flalign}
From \eqref{4.19} and \eqref{4.21}, it follows that for
\begin{flalign}
&\||J_{A}|^{s_{0}}u\|_{L_{t}^{\infty}((1,T);L_{x}^{2}(\mathbb{R}^{3}))}\nonumber\\
\leqslant&C_{1}\|u_{0}\|_{\Sigma_{s_{0}}}+C_{s_{0}}(\||J_{A}|^{\frac{3}{2}-\varepsilon}u\|_{L^{\infty}_{t}L^{2}_{x}}+\||J_{A}|u\|_{L^{\infty}_{t}L^{2}_{x}})\nonumber\\
+&C_{2}\|u_{0}\|^{(p-1)\cdot (1-\frac{3}{2s_{0}})}_{L_{x}^{2}}\cdot(\||J_{A}|^{s_{0}}u\|_{L_{t}^{\infty}L_{x}^{2}}+\||J_{A}|^{\frac{3}{2}-\varepsilon}u\|_{L_{t}^{\infty}L_{x}^{2}})^{\frac{3(p-1)}{2s_{0}}}
\nonumber\\
\cdot&(\||J_{A}|^{\frac{3}{2}-\varepsilon}u\|_{L_{t}^{\infty}L_{x}^{2}}+\||J_{A}|^{s_{0}}u\|_{L_{t}^{\infty}L_{x}^{2}}+\||J_{A}|u\|_{L_{t}^{\infty}L_{x}^{2}})\nonumber\\
\leqslant&C_{1}\|u_{0}\|_{\Sigma_{s_{0}}}+C_{s_{0}}(2\|u_{0}\|_{L^{2}}+\frac{1}{2}\||J_{A}|^{s_{0}}u\|_{L_{t}^{\infty}L_{x}^{2}})\nonumber\\
+&C_{2}\|u_{0}\|^{(p-1)(1-\frac{3}{2s_{0}})}_{L_{x}^{2}}(2\|u_{0}\|_{L^{2}}+\frac{1}{2}\||J_{A}|^{s_{0}}u\|_{L_{t}^{\infty}L_{x}^{2}})^{\frac{3(p-1)}{2s_{0}}}\nonumber\\
\cdot&(2\|u_{0}\|_{L^{2}}+\frac{1}{2}\||J_{A}|^{s_{0}}u\|_{L_{t}^{\infty}L_{x}^{2}})\nonumber\\
\leqslant&C_{1}\|u_{0}\|_{\Sigma_{s_{0}}}+C_{s_{0}}(2\|u_{0}\|_{L^{2}}+\frac{1}{2}\||J_{A}|^{s_{0}}u\|_{L_{t}^{\infty}L_{x}^{2}})\nonumber\\
\cdot&(1+\|u_{0}\|^{(p-1)(1-\frac{3}{2s_{0}})}_{L_{x}^{2}}(2\|u_{0}\|_{L^{2}}+\frac{1}{2}\||J_{A}|^{s_{0}}u\|_{L_{t}^{\infty}L_{x}^{2}})^{\frac{3(p-1)}{2s_{0}}})\label{4.22}
\end{flalign}
on any interval $(1,T)$ with a constant $C_{s_{0}}$ independent of $T$. Hence the proof of \eqref{4.20} follows by a standard continuity argument provided that we fix $\|u_{0}\|_{\Sigma_{s_{0}}}$  sufficiently small (also see the proof of Theorem 1.1 in \cite{CGV}).\\
Combining Lemma \ref{lemma2.5}, Lemma \ref{lemma4.2} and estimate \eqref{4.20}, we obtain that for some $s>s_{0}$
\begin{flalign}
\|u\|_{L^{\infty}_{x}(\mathbb{R}^{3})}&=\|M(-t)u\|_{L^{\infty}_{x}(\mathbb{R}^{3})}\nonumber\\
&\leqslant C\|M(-t)u\|_{L^{2}_{x}(\mathbb{R}^{3})}^{1-\frac{3}{2s_{0}}}\cdot\|M(-t)u\|_{\dot{H}^{s_{0}}(\mathbb{R}^{3})}^{\frac{3}{2s_{0}}}\nonumber\\
&\leqslant Ct^{-\frac{3}{2}}\cdot\|u\|_{L^{2}_{x}(\mathbb{R}^{3})}^{1-\frac{3}{2s_{0}}}\cdot\||J|^{s_{0}}u\|_{L^{2}_{x}(\mathbb{R}^{3})}^{\frac{3}{2s_{0}}}\nonumber\\
&\leqslant Ct^{(\varepsilon-\frac{3}{2})\cdot\frac{3}{2s_{0}}}\cdot\|u\|_{L^{2}_{x}(\mathbb{R}^{3})}^{1-\frac{3}{2s_{0}}}\cdot(\||J_{A}|^{\frac{3}{2}-\varepsilon}u\|_{L^{2}_{x}L_{t}^{\infty}}+\||J_{A}|^{s_{0}}u\|_{L^{2}_{x}L_{t}^{\infty}})^{\frac{3}{2s_{0}}}\nonumber\\
&\leqslant Ct^{(\varepsilon-\frac{3}{2})\cdot\frac{3}{2s_{0}}}\cdot\|u_{0}\|_{\Sigma_{s}}.\label{4.23}
\end{flalign}
That is, the proof of Theorem \ref{theorem1.1} is completed.

\section*{Appendix A. Global well-posedness}
\setcounter{equation}{0}
\setcounter{subsection}{0}
\setcounter{theorem}{0}
\setcounter{remark}{0}
\renewcommand{\theequation}{A.\arabic{equation}}
\renewcommand{\thesubsection}{A.\arabic{subsection}}
\renewcommand{\thetheorem}{A.\arabic{theorem}}
\renewcommand{\theremark}{A.\arabic{remark}}

In this appendix, we prove that the nonlinear magnetic Schr\"{o}dinger initial value problem is globally well-posed in $L^{\infty}_{t,x}((1,\infty)\times\mathbb{R}^{3})$ for $u_{0}\in\Sigma_{s}(\mathbb{R}^{3})$ ($s>\frac{3}{2}$).
\begin{theorem}\label{theorem A.1}
Let $A(x)$ satisfy hypothesis of Theorem \ref{theorem1.1}, $\frac{5}{3}<p<5$, and $|J_{A}|^{s}u\in L_{t}^{\infty}((1,\infty);L_{x}^{2}(\mathbb{R}^{3}))$ be a solution of equation \eqref{4.18}. If $u_{0}\in L^{2}_{x}(\mathbb{R}^{3})$, $|J_{A}|^{s}u(1)\in L^{2}_{x}(\mathbb{R}^{3})$ and $\|u_{0}\|_{L^{2}_{x}}$ is small enough, then the nonlinear Schr\"odinger equation \eqref{4.18} is globally well-posed.
\end{theorem}
\begin{proof}
Applying the Contraction Mapping Principle, we easily obtain that the solution $|J_{A}|^{s}u$ is locally well-posed in $L^{\infty}_tL^2_x$. Furthermore, according Strichartz estimates to nonlinear Schr\"odinger equation \eqref{4.18}, we obtain that
\begin{flalign}
&\||J_{A}|^{s}u\|_{L_{t}^{\infty}((1,T);L_{x}^{2}(\mathbb{R}^{3}))}\nonumber\\
\leqslant&C_{1}\||J_{A}|^{s}u(1)\|_{L_{x}^{2}(\mathbb{R}^{3})}\nonumber\\
+&C_{s}\|t^{s-1}M(t)V(s)M(-t)u\|_{L_{t}^{4}((1,T);L_{x}^{1}(\mathbb{R}^{3}))}\nonumber\\
+&C_{2}\||J_{A}|^{s}(|u|^{p-1}u)\|_{L_{t}^{1}((1,T);L_{x}^{2}(\mathbb{R}^{3}))}.\label{A.1}
\end{flalign}
Then, repeating the previous process in Lemma \ref{lemma4.3} and Lemma \ref{lemma4.4}, it follows that for $\forall s\in(\frac{3}{2},\frac{5}{3})$ and $\frac{5}{3}<p<5$
\begin{flalign}
&\||J_{A}|^{s}u\|_{L_{t}^{\infty}((1,T);L_{x}^{2}(\mathbb{R}^{3}))}\nonumber\\
\leqslant&C_{1}\||J_{A}|^{s}u(1)\|_{L_{x}^{2}}+C_{s}(\||J_{A}|^{\frac{3}{2}-\varepsilon}u\|_{L^{\infty}_{t}L^{2}_{x}}+\||J_{A}|u\|_{L^{\infty}_{t}L^{2}_{x}})\nonumber\\
+&C_{2}\|u_{0}\|^{(p-1)\cdot (1-\frac{3}{2s})}_{L_{x}^{2}}\cdot(\||J_{A}|^{s}u\|_{L_{t}^{\infty}L_{x}^{2}}+\||J_{A}|^{\frac{3}{2}-\varepsilon}u\|_{L_{t}^{\infty}L_{x}^{2}})^{\frac{3(p-1)}{2s}}
\nonumber\\
\cdot&(\||J_{A}|^{\frac{3}{2}-\varepsilon}u\|_{L_{t}^{\infty}L_{x}^{2}}+\||J_{A}|^{s}u\|_{L_{t}^{\infty}L_{x}^{2}}+\||J_{A}|u\|_{L_{t}^{\infty}L_{x}^{2}}).\label{A.2}
\end{flalign}
Therefore, let $\|u_{0}\|_{L^{2}_{x}(\mathbb{R}^{3})}$ be small enough, we obtain that
\begin{equation}\label{A.3}
\||J_{A}|^{s}u\|_{L_{t}^{\infty}((1,T);L_{x}^{2}(\mathbb{R}^{3}))}\leqslant C\||J_{A}|^{s}u(1)\|_{L_{x}^{2}}.
\end{equation}	
Next, when $t\in(T,2T)$, we can directly obtain that
\begin{equation}\label{A.4}
\||J_{A}|^{s}u\|_{L_{t}^{\infty}((T,2T);L_{x}^{2}(\mathbb{R}^{3}))}\leqslant C_{1}\||J_{A}|^{s}u(T)\|_{L_{x}^{2}}\leqslant C_{2}\||J_{A}|^{s}u(1)\|_{L_{x}^{2}}.
\end{equation}	
We can continue to do it for $t\in(1,T),(T,2T),(2T,3T)...$ until $t=\infty$, that is to say, \eqref{4.18} is globally well-posed for $|J_{A}|^{s}u\in L_{t}^{\infty}((1,\infty);L_{x}^{2}(\mathbb{R}^{3}))$.
\end{proof}

\begin{remark}\label{remark A.1}
It is well known that the solution to nonlinear Schr\"odinger equation \eqref{1.1}	conservers energy, therefore, we have the initial value problem in \eqref{1.1} is globally well-posed in $\dot{H}^{1}(\mathbb{R}^{3})$. In fact,
\begin{equation*}
\|u\|_{\dot{H}^{1}(\mathbb{R}^{3})}^{2}=\|\nabla u\|_{L^{2}_{x}}^{2}\lesssim E(u)+\|u\|^{p+1}_{L^{p+1}}
=E(u_{0})+\|u\|^{p+1}_{L^{p+1}_{x}},
\end{equation*}
where
\begin{equation}
\label{A.5}
E(u_{0})=\|\nabla u_{0}\|_{L^{2}_{x}}^{2}+\|u_{0}\|^{p+1}_{L^{p+1}_{x}}, \end{equation}
and
\begin{flalign}
&\|u\|^{p+1}_{L^{p+1}_{x}(\mathbb{R}^{3})}\nonumber\\
=&\|M(-t)u\|^{p+1}_{L^{p+1}_{x}(\mathbb{R}^{3})}\nonumber\\
\leqslant&C\|M(-t)u\|_{L^{2}_{x}(\mathbb{R}^{3})}^{(1-3(\frac{1}{2}-\frac{1}{p+1}))(p+1)}\cdot\|\nabla(M(-t)u)\|_{L^{2}_{x}(\mathbb{R}^{3})}^{3(\frac{1}{2}-\frac{1}{p+1})(p+1)}\nonumber\\
\leqslant&C\|u_{0}\|_{L^{2}_{x}(\mathbb{R}^{3})}^{3-\frac{p+1}{2}}\cdot t^{-\frac{3}{2}(p-1)}\cdot\||J_{A}|u\|_{L^{2}_{x}(\mathbb{R}^{3})}^{\frac{3}{2}(p-1)}\label{A.6}.
\end{flalign}
Combining \eqref{A.5} and \eqref{A.6}, we obtain that for $1<p<5$
\begin{flalign}
&\|u\|_{\dot{H}^{1}(\mathbb{R}^{3})}^{2}\nonumber\\
\leqslant&C\|\nabla u_{0}\|_{L^{2}_{x}}^{2}+\|u_{0}\|_{\dot{H}^{1}_{x}}^{p+1}+\|u_{0}\|_{\Sigma_{1}}^{p+1}\nonumber\\
\leqslant&C\|u_{0}\|_{\Sigma_{1}}^{p+1}.\label{A.7}
\end{flalign}
That is to say, it make sense to inequality \eqref{3.10} in Lemma \ref{lemma3.2}.
\end{remark}
\begin{remark}\label{remark A.2}
Furthermore, again applying estimate \eqref{4.23} and the global well-posedness of $|J_{A}|^{s}u$, we have
\begin{flalign}
&\|u\|_{L^{\infty}_{t,x}((1,\infty)\times\mathbb{R}^{3})}\nonumber\\
=&\|M(-t)u\|_{L^{\infty}_{t,x}((1,\infty)\times\mathbb{R}^{3})}\nonumber\\
\leqslant& C\|t^{(\varepsilon-\frac{3}{2})\cdot\frac{3}{2s}}\|_{L^{\infty}_{t}(1,\infty)}\cdot\|u\|_{L^{2}_{x}(\mathbb{R}^{3})}^{1-\frac{3}{2s}}\cdot(\||J_{A}|^{\frac{3}{2}-\varepsilon}u\|_{L^{2}_{x}L_{t}^{\infty}}+\||J_{A}|^{s}u\|_{L^{2}_{x}L_{t}^{\infty}})^{\frac{3}{2s}}\nonumber\\
\leqslant& C\|t^{(\varepsilon-\frac{3}{2})\cdot\frac{3}{2s}}\|_{L^{\infty}_{t}(1,\infty)}\cdot\||J_{A}|^{s}u(1)\|_{L_{x}^{2}}\nonumber\\
\leqslant& C\||J_{A}|^{s}u(1)\|_{L_{x}^{2}}.\label{A.8}
\end{flalign}
Hence, we obtain that the initial value problem of \eqref{1.1} is globally well-posed in $L^{\infty}_{t,x}((1,\infty)\times\mathbb{R}^{3})$. In other words, it make sense to our decay result \eqref{1.4} in Theorem \ref{theorem1.1}.
\end{remark}

\noindent\textbf{Acknowledgements:} We would like to thank the reviewer for the many useful comments.

\end{document}